\documentclass[a4paper]{amsart}
\usepackage[foot]{amsaddr}

\usepackage[latin1]{inputenc}
\usepackage{amssymb,amsmath,amsthm}
\usepackage{amsmath}
\usepackage{amsfonts}
\usepackage{enumitem}
\usepackage{booktabs}
\usepackage{ifthen}
\usepackage{tikz}
\usetikzlibrary{math,calc}
\usepackage{url}
\usepackage{hyphenat}
\usepackage{hyperref}
\usepackage{graphicx}
\usepackage{calc,etoolbox}
\usepackage{rotating}
\usepackage[para]{threeparttable}
\usepackage{caption}
\usepackage{multicol}
\usepackage{multirow}
\usepackage{subfig}
\usepackage{caption}

\theoremstyle{plain}
\newtheorem{theorem}{Theorem}[section]
\newtheorem{proposition}[theorem]{Proposition}
\newtheorem{lemma}[theorem]{Lemma}
\newtheorem{corollary}[theorem]{Corollary}

\theoremstyle{definition}
\newtheorem{definition}[theorem]{Definition}

\newtheorem{observation}[theorem]{Observation}
\newtheorem{claim}{Claim}
\numberwithin{claim}{theorem}
\newenvironment{pfclaim}[1][Proof]{\begin{trivlist}
\item[\hskip \labelsep {\textit{{#1}.}}]} {{\footnotesize $\blacksquare$}\end{trivlist}}

\numberwithin{figure}{section}
\numberwithin{table}{section}

\newcommand{\IN}{\ensuremath{\mathbb{N}}}

\newcommand{\minorant}{\ensuremath{\leqq}}
\newcommand{\minor}{\ensuremath{\leq}}
\newcommand{\nset}[1]{\ensuremath{[{#1}]}}
\newcommand{\card}[1]{\ensuremath{\lvert{#1}\rvert}}
\newcommand{\gendefault}{}
\newcommand{\gen}[2][\gendefault]{\ensuremath{\langle{#2}\rangle_{#1}}}
\newcommand{\clonegen}[1]{\gen[]{#1}}
\newcommand{\vect}[1]{\ensuremath{\mathbf{#1}}}
\newcommand{\lhs}{\hspace{2em}&\hspace{-2em}}

\newcommand{\clIntVal}[3]{\ensuremath{#1_{\ifthenelse{\equal{#2}{}}{\mathord{*}}{#2}\ifthenelse{\equal{#3}{}}{\mathord{*}}{#3}}}}

\newcommand{\clAll}{\ensuremath{\mathsf{\Omega}}}

\newcommand{\clOX}{\ensuremath{{\clIntVal{\clAll}{0}{}}}}
\newcommand{\clIX}{\ensuremath{{\clIntVal{\clAll}{1}{}}}}
\newcommand{\clXO}{\ensuremath{{\clIntVal{\clAll}{}{0}}}}
\newcommand{\clXI}{\ensuremath{{\clIntVal{\clAll}{}{1}}}}

\newcommand{\clOXCI}{\ensuremath{\clOX \cup \clVaki}}

\newcommand{\clIXCO}{\ensuremath{\clIX \cup \clVako}}

\newcommand{\clXOCI}{\ensuremath{\clXO \cup \clVaki}}

\newcommand{\clXICO}{\ensuremath{\clXI \cup \clVako}}

\newcommand{\clOI}{\ensuremath{\clIntVal{\clAll}{0}{1}}}

\newcommand{\clS}{\ensuremath{\mathsf{S}}}
\newcommand{\clSc}{\ensuremath{\clIntVal{\clS}{0}{1}}}

\newcommand{\clSM}{\ensuremath{\mathsf{SM}}}

\newcommand{\clM}{\ensuremath{\mathsf{M}}}
\newcommand{\clMo}{\ensuremath{\clIntVal{\clM}{0}{}}}
\newcommand{\clMi}{\ensuremath{\clIntVal{\clM}{}{1}}}
\newcommand{\clMc}{\ensuremath{\clIntVal{\clM}{0}{1}}}
\newcommand{\clMneg}{\ensuremath{{\overline{\clM}}}}
\newcommand{\clMoneg}{\ensuremath{\clIntVal{\clMneg}{1}{}}}
\newcommand{\clMineg}{\ensuremath{\clIntVal{\clMneg}{}{0}}}
\newcommand{\clMcneg}{\ensuremath{\clIntVal{\clMneg}{1}{0}}}
\newcommand{\clUk}[1]{\ensuremath{\mathsf{U}^{#1}}}

\newcommand{\clTcUk}[1]{\ensuremath{\clIntVal{\clUk{#1}}{0}{1}}}

\newcommand{\clMUk}[1]{\ensuremath{\mathsf{M}\clUk{#1}}}
\newcommand{\clMcUk}[1]{\ensuremath{\clIntVal{\clMUk{#1}}{0}{1}}}
\newcommand{\clWk}[1]{\ensuremath{\mathsf{W}^{#1}}}

\newcommand{\clTcWk}[1]{\ensuremath{\clIntVal{\clWk{#1}}{0}{1}}}

\newcommand{\clMWk}[1]{\ensuremath{\mathsf{M}\clWk{#1}}}
\newcommand{\clMcWk}[1]{\ensuremath{\clIntVal{\clMWk{#1}}{0}{1}}}

\newcommand{\clVak}{\ensuremath{\mathsf{C}}}
\newcommand{\clVaka}[1]{\ensuremath{\clVak_{#1}}}
\newcommand{\clVakaAB}[2]{\ensuremath{\clVak^{#2}_{#1}}}
\newcommand{\clVako}{\ensuremath{\clVaka{0}}}
\newcommand{\clVaki}{\ensuremath{\clVaka{1}}}
\newcommand{\clEmpty}{\ensuremath{\mathsf{\emptyset}}}

\newcommand{\clL}{\ensuremath{\mathsf{L}}}
\newcommand{\clLo}{\ensuremath{\clIntVal{\clL}{0}{}}}
\newcommand{\clLi}{\ensuremath{\clIntVal{\clL}{}{1}}}
\newcommand{\clLc}{\ensuremath{\clIntVal{\clL}{0}{1}}}
\newcommand{\clLS}{\ensuremath{\mathsf{LS}}}
\newcommand{\clLambda}{\ensuremath{\mathsf{\Lambda}}}
\newcommand{\clV}{\ensuremath{\mathsf{V}}}
\newcommand{\clLambdac}{\ensuremath{\clIntVal{\clLambda}{0}{1}}}
\newcommand{\clVc}{\ensuremath{\clIntVal{\clV}{0}{1}}}
\newcommand{\clLambdao}{\ensuremath{\clIntVal{\clLambda}{0}{}}}
\newcommand{\clVo}{\ensuremath{\clIntVal{\clV}{0}{}}}
\newcommand{\clLambdai}{\ensuremath{\clIntVal{\clLambda}{}{1}}}
\newcommand{\clVi}{\ensuremath{\clIntVal{\clV}{}{1}}}
\newcommand{\clOmegaOne}{\ensuremath{\clAll(1)}}
\newcommand{\clIstar}{\ensuremath{\mathsf{I}^{*}}}
\newcommand{\clIa}[1]{\ensuremath{\mathsf{I}_{#1}}}
\newcommand{\clI}{\ensuremath{\mathsf{I}}}
\newcommand{\clIo}{\ensuremath{\mathsf{I}_0}}
\newcommand{\clIi}{\ensuremath{\mathsf{I}_1}}
\newcommand{\clIc}{\ensuremath{\mathsf{J}}}
\newcommand{\clUinf}{\ensuremath{\clUk{\infty}}}
\newcommand{\clWinf}{\ensuremath{\clWk{\infty}}}

\newcommand{\vak}[1]{\ensuremath{\mathrm{c}_{#1}}}
\newcommand{\id}{\ensuremath{\mathrm{id}}}

\DeclareMathOperator{\pr}{pr}
\newcommand{\arity}[1]{\ensuremath{\mathrm{ar}({#1})}}
\newcommand{\Alt}{\ensuremath{\mathrm{Alt}}}

\newcommand{\closys}[1]{\ensuremath{\mathcal{L}_{#1}}}
\newcommand{\clProj}[1]{\ensuremath{\mathsf{J}_{#1}}}

\DeclareMathOperator{\range}{Im}

\makeatletter
\newcommand{\displaybump}{\hbox to \@totalleftmargin{\hfil}}
\makeatother

\begin{document}
\title[Clonoids of Boolean functions]{Clonoids of Boolean functions with essentially unary, linear, semilattice, or 0- or 1-separating source and target clones}

\author{Erkko Lehtonen}

\address{%
    Department of Mathematics \\
    Khalifa University of Science and Technology \\
    P.O. Box 127788 \\
    Abu Dhabi \\
    United Arab Emirates
}

\date{\today}

\begin{abstract}
Extending Sparks's theorem, we determine the cardinality of the lattice of $(C_1,C_2)$\hyp{}clonoids of Boolean functions for certain pairs $(C_1,C_2)$ of clones of essentially unary, linear, or $0$- or $1$\hyp{}separating functions or semilattice operations.
When such a $(C_1,C_2)$\hyp{}clonoid lattice is uncountable, the proof is in most cases based on exhibiting a countably infinite family of functions with the property that distinct subsets thereof always generate distinct $(C_1,C_2)$\hyp{}clonoids.
In the cases when the lattice is finite, we enumerate the corresponding $(C_1,C_2)$\hyp{}clonoids.
We also provide a summary of the known results on cardinalities of $(C_1,C_2)$\hyp{}clonoid lattices of Boolean functions.
\end{abstract}

\maketitle


\section{Introduction}

Composition is a most fundamental operation on (multivariate) functions.
This notion can be extended to sets of functions in a natural way:
the composition of sets $F$ and $G$ of multivariate functions is the set $FG$ of all composite functions of the form $f ( g_1, \dots, g_n)$, where $f \in F$ and $g_1, \dots, g_n \in G$.

Clones, minions, and clonoids are sets of multivariate functions from a set $A$ to a set $B$, or operations on a set $A$,
with closure properties that can be expressed in terms of function class composition.
Namely, a \emph{clone} on $A$ is a set $C$ of operations on $A$ that contains all projections and $C C \subseteq C$.
For fixed clones $C_1$ and $C_2$ on sets $A$ and $B$, respectively, a set $K$ of multivariate functions from $A$ to $B$ is a \emph{$(C_1,C_2)$\hyp{}clonoid} if $K C_1 \subseteq K$ and $C_2 K \subseteq K$.
Denoting by $\clProj{A}$ and $\clProj{B}$ the clones of all projections on $A$ and $B$, respectively, we obtain as a special case the $(\clProj{A},\clProj{B})$\hyp{}clonoids which are also called \emph{minions} or \emph{minor\hyp{}closed classes}.

The terminology we use here is relatively modern.
To the best of the author's knowledge, the term ``clone'' was first used in the universal\hyp{}algebraic sense in the 1965 monograph of Cohn~\cite{Cohn}, who attributed it to Philip Hall.
The term ``clonoid'' was introduced in the 2016 paper by Aichinger and Mayr \cite{AicMay}, and ``minion'' was coined by Opr\v{s}al around the year 2018 (see \cite[Definition~2.20]{BarBulKroOpr}, \cite{BulKroOpr}).
It should, however, be noted that these concepts have appeared in the literature much earlier.
For further information and general background on universal algebra and clones, see, e.g., the monographs by Bergman~\cite{Bergman} and Szendrei~\cite{Szendrei}.

In universal algebra, clones arise naturally as sets of term operations of algebras, or as sets of polymorphisms of relations.
Minions, in turn, arise as term operations induced by terms of height $1$, or as sets of polymorphisms of relation pairs (see Pippenger~\cite{Pippenger}).
If in a relation pair $(R,S)$, the relations $R$ and $S$ are invariants of clones $C_1$ and $C_2$, respectively, the set of polymorphisms of $(R,S)$ is a $(C_1,C_2)$\hyp{}clonoid (see Couceiro and Foldes \cite{CouFol-2005,CouFol-2009}).

In theoretical computer science,
constraint satisfaction problems (CSP) are a central topic in computational complexity theory.
Universal\hyp{}algebraic tools, including clones and polymorphisms, have proved successful in the analysis of computational complexity of CSPs.
In particular, minions arise in the context of a new variant called promise CSP\@.
For further details, see the excellent survey article by Barto et al.\ \cite{BarBulKroOpr}.

One of the main open problems in universal algebra is the classification of clones on finite sets with at least three elements.
(The problem of classifying the clones on a one\hyp{}element set is trivial.
The clones on a two\hyp{}element set were described by Post \cite{Post}.)
On account of the growing interest in minions and clonoids, we are inevitably led to the classification problem of $(C_1,C_2)$\hyp{}clonoids.
Such classification results for certain types of clone pairs $(C_1,C_2)$ have appeared in the literature,
for example,
Fioravanti~\cite{Fioravanti-AU,Fioravanti-IJAC},
Kreinecker~\cite{Kreinecker},
and
Mayr and Wynne~\cite{MayWyn}.

Considering that the clones on the two\hyp{}element set $\{0,1\}$ are well known, the present author initiated the effort of systematically counting and enumerating all $(C_1,C_2)$\hyp{}clonoids, for each pair $(C_1,C_2)$ of clones on $\{0,1\}$.
An opportune starting point towards this goal is the following remarkable result due to Sparks.
Here, $\closys{(C_1,C_2)}$ denotes the lattice of $(C_1,C_2)$\hyp{}clonoids.
An operation $f \colon A^n \to A$ is a \emph{near\hyp{}unanimity operation} if it satisfies all identities of the form $f(x, \dots, x, y, x, \dots, x) \approx x$, where the single $y$ is at any argument position.
A ternary near\hyp{}unanimity operation is called a \emph{majority operation}.
A \emph{Mal'cev operation} is a ternary operation that satisfies the identities $f(x,x,y) \approx f(y,x,x) \approx y$.

\begin{theorem}[{Sparks~\cite[Theorem~1.3]{Sparks-2019}}]
\label{thm:Sparks}
Let $A$ be a finite set with $\card{A} > 1$, and let $B = \{0,1\}$.
Let $C$ be a clone on $B$.
Then the following statements hold.
\begin{enumerate}[label={\upshape(\roman*)}]
\item\label{thm:Sparks:finite} $\closys{(\clProj{A},C)}$ is finite if and only if $C$ contains a near\hyp{}unanimity operation.
\item\label{thm:Sparks:countable} $\closys{(\clProj{A},C)}$ is countably infinite if and only if $C$ contains a Mal'cev operation but no majority operation.
\item\label{thm:Sparks:uncountable} $\closys{(\clProj{A},C)}$ has the cardinality of the continuum if and only if $C$ contains neither a near\hyp{}unanimity operation nor a Mal'cev operation.
\end{enumerate}
\end{theorem}

It should be noted that Theorem~\ref{thm:Sparks} and its published proof only reveal the cardinality of the lattice $\closys{(C_1,C_2)}$ of $(C_1,C_2)$\hyp{}clonoids; they do not actually describe the clonoids themselves.
Moreover, the source clone $C_1$ is always assumed to be the clone of projections.
In a series of papers of the present author \cite{CouLeh-Lcstability,Lehtonen-SM,Lehtonen-nu,Lehtonen-discmono}, Theorem~\ref{thm:Sparks} was extended and sharpened.
For many pairs $(C_1,C_2)$ of clones on $\{0,1\}$, the cardinality of the lattice $\closys{(C_1,C_2)}$ was determined and the $(C_1,C_2)$\hyp{}clonoids were described whenever the lattice was shown to be finite or countably infinite.
See Section~\ref{sec:summary} and Table~\ref{table:card} for a more detailed summary of our earlier work.

In the current paper, we take a few modest additional steps towards classifying $(C_1,C_2)$\hyp{}clonoids of Boolean functions.
We focus on certain pairs $(C_1,C_2)$ of clones, where the source $C_1$ and the target $C_2$ are clones of essentially unary functions ($\clIo$, $\clIi$, $\clIstar$, $\clOmegaOne$), linear functions ($\clL$), semilattice operations ($\clVo$, $\clV$, $\clLambdai$, $\clLambda$), or $0$- or $1$\hyp{}separating functions ($\clMcUk{\infty}$, $\clMcWk{\infty}$, $\clUk{2}$, $\clWk{2}$).
(For the definitions of, and the notation for, the clones on $\{0,1\}$, see Subsection~\ref{subsec:Bf} and Figure~\ref{fig:Post}.)
Our main results are the following:
\begin{itemize}
\item Theorem~\ref{thm:uncountable}:
For all clones $C_1$ and $C_2$ such that $C_1 \subseteq K_1$ and $C_2 \subseteq K_2$, where
\[
\begin{split}
\displaybump
(K_1,K_2) \in
\{
&
(\clOmegaOne, \clOmegaOne),
(\clOmegaOne, \clLambda),
(\clOmegaOne, \clV),
(\clIstar, \clUinf),
(\clIstar, \clWinf),
\\
&(\clIo, \clUinf),
(\clIi, \clUinf),
(\clIo, \clWinf),
(\clIi, \clWinf),
(\clL, \clOmegaOne),
\\
&
(\clLambda, \clLambda),
(\clLambda, \clV),
(\clV, \clLambda),
(\clV, \clV),
(\clLambda, \clOmegaOne),
(\clV, \clOmegaOne)
\},
\end{split}
\]
there are an uncountable infinitude of $(C_1,C_2)$\hyp{}clonoids.
The proof is based on exhibiting a countably infinite set $F$ of functions with the property that $f$ is in the $(K_1,K_2)$\hyp{}clonoid generated by $F$ if and only if $f \in F$.
Since $F$ has an uncountable infinity of subsets, the result follows.

\item Theorem~\ref{thm:U2Uinf}:
For all clones $C_1$ and $C_2$ such that $C_1 \subseteq K_1$ and $C_2 \subseteq K_2$ for some $K_1 \in \{\clUk{2}, \clWk{2}\}$ and $K_2 \in \{\clUk{\infty}, \clWk{\infty}\}$, there are an uncountable infinitude of $(C_1,C_2)$\hyp{}clonoids.
The proof is based on the universality of the homomorphism order of upwards closed loopless hypergraphs
and an application of a result on so\hyp{}called $\clUk{k}$- and $\clWk{k}$\hyp{}minors of Boolean functions due to Ne\v{s}et\v{r}il and the present author \cite{LehNes-clique}.

\item Theorem~\ref{thm:finite}:
For all clones $C_1$ and $C_2$ such that $C_1 \supseteq K_1$ and $C_2 \supseteq K_2$ for some
$K_1 \in \{\clI, \clVo, \clLambdai\}$
and
$K_2 \in \{\clMcUk{\infty}, \clMcWk{\infty} \}$,
there are only finitely many $(C_1,C_2)$\hyp{}clonoids.

\item
Additionally, we establish in Section~\ref{sec:helpful} some general facts about relationships between $(C_1,C_2)$- and $(C'_1,C'_2)$\hyp{}clonoids when the clones $C_1$ and $C'_1$ or $C_2$ and $C'_2$ are duals of each other, or when the target clone $C'_2$ is obtained from $C_2$ by adding constant functions or negated projections.
These help us find the cardinalities of the $(C_1,C_2)$\hyp{}clonoid lattices for a few more clone pairs $(C_1,C_2)$.
\end{itemize}

In the final Section~\ref{sec:summary}, we summarize the known facts about the $(C_1,C_2)$\hyp{}clonoid lattices of Boolean functions, giving references to the relevant results in the literature.
We also highlight the clone pairs $(C_1,C_2)$ for which the cardinality of $\closys{(C_1,C_2)}$ is still unknown, which gives us directions for further research.


\section{Preliminaries}

\subsection{General}

The sets of nonnegative integers and positive integers are denoted by $\IN$ and $\IN^{+}$, respectively.
For $n \in \IN$, let $\nset{n} := \{ \, i \in \IN^{+} \mid 1 \leq i \leq n \, \}$. 
We denote tuples by bold letters and their components by the corresponding italic letters, e.g., $\vect{a} = (a_1, \dots, a_n)$.

\subsection{Clones and clonoids}

Let $A$ and $B$ be nonempty sets.
A \emph{function of several arguments} from $A$ to $B$ is a mapping $f \colon A^n \to B$ for some positive integer $n$ called the \emph{arity} of $f$.
We denote by $\mathcal{F}_{AB}^{(n)}$ the set of all $n$\hyp{}ary functions from $A$ to $B$ and we let $\mathcal{F}_{AB} := \bigcup_{n \in \IN^{+}} \mathcal{F}_{AB}^{(n)}$.
In the case when $A = B$, we speak of \emph{operations} on $A$, and we write $\mathcal{O}_A^{(n)}$ and $\mathcal{O}_A$ for $\mathcal{F}_{AA}^{(n)}$ and $\mathcal{F}_{AA}$, respectively.
For $C \subseteq \mathcal{F}_{AB}$ and $n \in \IN$, the \emph{$n$\hyp{}ary part} of $C$ is $C^{(n)} := C \cap \mathcal{F}_{AB}^{(n)}$.

If $f \in \mathcal{F}_{BC}^{(n)}$ and $g_1, \dots, g_n \in \mathcal{F}_{AB}^{(m)}$, then the \emph{composition} of $f$ with $(g_1, \dots, g_n)$, denoted by $f(g_1, \dots, g_n)$ is a function in $\mathcal{F}_{AC}^{(m)}$ and is defined by the rule
\[
f(g_1, \dots, g_n)(\vect{a}) :=
f(g_1(\vect{a}), \dots, g_n(\vect{a})),
\]
for all $\vect{a} \in A^m$.
The notion of composition extends to function classes.
For $F \subseteq \mathcal{F}_{BC}$ and $G \subseteq \mathcal{F}_{AB}$, the \emph{composition} of $F$ with $G$, denoted by $FG$, is
the set of all composite functions of the form $f(g_1, \dots, g_n)$, where, for some $n, m \in \IN^{+}$, $f \in F^{(n)}$ and $g_1, \dots, g_n \in G^{(m)}$.
Function class composition is monotone, i.e., if $F, F' \subseteq \mathcal{F}_{BC}$, $G, G' \subseteq \mathcal{F}_{AB}$ satisfy $F \subseteq F'$ and $G \subseteq G'$, then $F G \subseteq F' G'$.

For $n, i \in \IN$ with $1 \leq i \leq n$, the $i$\hyp{}th $n$\hyp{}ary \emph{projection} on $A$ is the operation $\pr_i^{(n)} \colon A^n \to A$, $(a_1, \dots, a_n) \mapsto a_i$.
A \emph{clone} on $A$ is a set $C \subseteq \mathcal{O}_A$ that is closed under composition and contains all projections, in symbols, $C C \subseteq C$ and $\clProj{A} \subseteq C$.
The clones on $A$ form a closure system on $\mathcal{O}_A$, and the clone generated by a set $F \subseteq \mathcal{O}_A$, i.e., the least clone containing $F$, is denoted by $\clonegen{F}$.

Let $A$ and $B$ be arbitrary nonempty sets, and let $C_1$ be a clone on $A$ (the \emph{source clone}) and let $C_2$ be a clone on $B$ (the \emph{target clone}).
A set $K \subseteq \mathcal{F}_{AB}$ is called a \emph{$(C_1,C_2)$\hyp{}clonoid} if $K C_1 \subseteq K$ and $C_2 K \subseteq K$
($K$ is stable under right composition with $C_1$ and under left composition with $C_2$).
The set $\closys{(C_1,C_2)}$ of all $(C_1,C_2)$\hyp{}clonoids forms a closure system on $\mathcal{F}_{AB}$,
and the least $(C_1,C_2)$\hyp{}clonoid containing a set $F \subseteq \mathcal{F}_{AB}$ is denoted by $\gen[(C_1,C_2)]{F}$.
A $(\clProj{A},\clProj{B})$\hyp{}clonoid is called a \emph{minion} or a \emph{minor\hyp{}closed class}.

We review here a few useful facts about clones, clonoids, and function class composition.

Although the composition of functions is associative, function class composition is not.

\begin{lemma}[{Couceiro, Foldes \cite[Associativity Lemma]{CouFol-2007,CouFol-2009}}]
\label{lem:Associativity}
Let $A$, $B$, $C$, and $D$ be arbitrary nonempty sets, and let $I \subseteq \mathcal{F}_{CD}$, $J \subseteq \mathcal{F}_{BC}$, $K \subseteq \mathcal{F}_{AB}$.
Then the following statements hold.
\begin{enumerate}[label=\upshape{(\roman*)}]
\item $(IJ)K \subseteq I(JK)$.
\item If $J$ is a minion, then $(IJ)K = I(JK)$.
\end{enumerate}
\end{lemma}

\begin{lemma}[{\cite[Lemma~2.16]{CouLeh-Lcstability}}]
\label{lem:clonmon}
Let $C_1$ and $C'_1$ be clones on $A$ and $C_2$ and $C'_2$ clones on $B$ such that $C_1 \subseteq C'_1$ and $C_2 \subseteq C'_2$.
Then every $(C'_1,C'_2)$\hyp{}clonoid is a $(C_1,C_2)$\hyp{}clonoid.
\end{lemma}

\begin{lemma}[{\cite[Lemma~2.5]{Lehtonen-SM}}]
Let $F \subseteq \mathcal{F}_{AB}$, and let $C_1$ and $C_2$ be clones on $A$ and $B$, respectively.
Then $\gen[(C_1,C_2)]{F} = C_2 ( F C_1 )$.
\end{lemma}

\subsection{Boolean functions}
\label{subsec:Bf}

Operations on $\{0,1\}$ are called \emph{Boolean functions}.
We introduce terminology and notation for classes of Boolean functions that will be used later.
Table~\ref{tab:Bfs} defines a few well\hyp{}known Boolean functions:
$\vak{0}$ and $\vak{1}$ (constant functions),
$\id$ (identity),
$\neg$ (negation),
$\wedge$ (conjunction),
$\vee$ (disjunction),
$+$ (addition modulo $2$).
Recall that $\pr_i^{(n)}$ denotes the $i$\hyp{}th $n$\hyp{}ary projection;
thus $\id = \pr_1^{(1)}$.
For $1 \leq i \leq n$, we also let $\neg_i^{(n)} := \neg(\pr_i^{(n)})$, the $i$\hyp{}th $n$\hyp{}ary \emph{negated projection}.

\begin{table}
\begin{tabular}[t]{c|cccc}
$x_1$ & $\vak{0}$ & $\vak{1}$ & $\id$ & $\neg$ \\
\hline
$0$ & $0$ & $1$ & $0$ & $1$ \\
$1$ & $0$ & $1$ & $1$ & $0$
\end{tabular}
\qquad\qquad
\begin{tabular}[t]{cc|ccc}
$x_1$ & $x_2$ & $\wedge$ & $\vee$ & $+$ \\
\hline
$0$ & $0$ & $0$ & $0$ & $0$ \\
$0$ & $1$ & $0$ & $1$ & $1$ \\
$1$ & $0$ & $0$ & $1$ & $1$ \\
$1$ & $1$ & $1$ & $1$ & $0$
\end{tabular}
\caption{Some well\hyp{}known Boolean functions.}
\label{tab:Bfs}
\end{table}

The \emph{complement} of $a \in \{0,1\}$ is $\overline{a} := 1 - a$.
The \emph{complement} of $\vect{a} = (a_1, \dots, a_n) \in \{0,1\}^n$ is $\overline{\vect{a}} := (\overline{a_1}, \dots, \overline{a_n})$.
We regard the set $\{0,1\}$ totally ordered by the natural order $0 < 1$, which induces the direct product order on $\{0,1\}^n$.
The poset $(\{0,1\}^n, \mathord{\leq})$ is a Boolean lattice, i.e., a complemented distributive lattice with least and greatest elements $\vect{0} = (0, \dots, 0)$ and $\vect{1} = (1, \dots, 1)$ and with the map $\vect{a} \mapsto \overline{\vect{a}}$ being the complementation.

The set of all Boolean functions is denoted by $\clAll$.
For $a, b \in \{0,1\}$, let
\begin{align*}
\clIntVal{\clAll}{a}{} &:= \{ \, f \in \clAll \mid f(0, \dots, 0) = a \, \}, \\
\clIntVal{\clAll}{}{b} &:= \{ \, f \in \clAll \mid f(1, \dots, 1) = b \, \}, \\
\clIntVal{\clAll}{a}{b} &:= \clIntVal{\clAll}{a}{} \cap \clIntVal{\clAll}{}{b}.
\end{align*}
Moreover, for any $K \subseteq \clAll$, let
$\clIntVal{K}{a}{} := K \cap \clIntVal{\clAll}{a}{}$,
$\clIntVal{K}{}{b} := K \cap \clIntVal{\clAll}{}{b}$,
$\clIntVal{K}{a}{b} := K \cap \clIntVal{\clAll}{a}{b}$.

For $x \in \{0,1\}$ and $n \in \IN^{+}$, the $n$\hyp{}ary \emph{constant function} taking value $x$ is $\vak{x}^{(n)} \colon \{0,1\}^n \to \{0,1\}$, $\vak{x}^{(n)}(\vect{a}) = x$ for all $\vect{a} \in \{0,1\}^n$.
We will omit the superscript indicating the arity when it is clear from the context or irrelevant.
We denote by $\clVak$ the set of all constant Boolean functions.
We use the shorthands $\clVako := \clIntVal{\clVak}{0}{0}$ and $\clVaki := \clIntVal{\clVak}{1}{1}$.

Let $f \in \clAll^{(n)}$.
The elements of $f^{-1}(1)$ and those of $f^{-1}(0)$ are the \emph{true points} and the \emph{false points} of $f$, respectively.
The \emph{negation} $\overline{f}$, the \emph{inner negation} $f^\mathrm{n}$, and the \emph{dual} $f^\mathrm{d}$ of $f$ are the $n$\hyp{}ary Boolean functions given by the rules
$\overline{f}(\vect{a}) := \overline{f(\vect{a})}$,
$f^\mathrm{n}(\vect{a}) := f(\overline{\vect{a}})$,
and
$f^\mathrm{d}(\vect{a}) := \overline{f(\overline{\vect{a}})}$,
for all $\vect{a} \in \{0,1\}^n$.
We can write these definitions using functional composition as
$\overline{f} := \neg(f)$,
$f^\mathrm{n} := f(\neg_1^{(n)}, \dots, \neg_n^{(n)})$,
and
$f^\mathrm{d} := \neg(f(\neg_1^{(n)}, \dots, \neg_n^{(n)}))$.
A Boolean function $f \in \clAll^{(n)}$ is \emph{self\hyp{}dual} if $f = f^\mathrm{d}$.
We denote by $\clS$ the set of all self\hyp{}dual functions.
For any $C \subseteq \clAll$, let
$\overline{C} := \{ \, \overline{f} \mid f \in C \, \}$,
$C^\mathrm{n} := \{ \, f^\mathrm{n} \mid f \in C \, \}$,
$C^\mathrm{d} := \{ \, f^\mathrm{d} \mid f \in C \, \}$.

A Boolean function $f \in \clAll^{(n)}$ is \emph{monotone} if for all $\vect{a}, \vect{b} \in \{0,1\}^n$, $\vect{a} \leq \vect{b}$ implies $f(\vect{a}) \leq f(\vect{b})$.
The set of all monotone functions is denoted by $\clM$.
We use the shorthand $\clSM := \clS \cap \clM$.

For $k \in \{ \, n \in \IN \mid n \geq 2 \, \} \cup \{\infty\}$,
a function $f \in \clAll^{(n)}$ is
\emph{$1$\hyp{}separating of rank $k$} if for all $T \subseteq f^{-1}(1)$ with $\card{T} \leq k$, it holds that $\bigwedge T \neq \vect{0}$,
and
$f$ is \emph{$0$\hyp{}separating of rank $k$} if for all $F \subseteq f^{-1}(0)$ with $\card{F} \leq k$, it holds that $\bigvee F \neq \vect{1}$.
We denote by $\clUk{k}$ and $\clWk{k}$ the set of all $1$\hyp{}separating functions of rank $k$ and the set of all $0$\hyp{}separating functions of rank $k$, respectively.
We use the shorthands $\clMUk{k} := \clM \cap \clUk{k}$ and $\clMWk{k} := \clM \cap \clWk{k}$.

Let $\clL := \clonegen{\mathord{+}, \vak{1}}$, the clone of \emph{linear functions};
$\clLambda := \clonegen{\mathord{\wedge}, \vak{0}, \vak{1}}$, the clone of conjunctions and constants;
and
$\clV := \clonegen{\mathord{\vee}, \vak{0}, \vak{1}}$, the clone of disjunctions and constants.

We denote the set of all projections, negated projections, and constant functions (the essentially at most unary functions) by $\clOmegaOne$;
the set of all projections and negated projections by $\clIstar$;
the set of all projections and constant functions by $\clI$;
the set of all projections and constant functions taking value $0$ by $\clIo$;
the set of all projections and constant functions taking value $1$ by $\clIi$;
and
the set of all projections by $\clIc$.
(Thus, $\clIo = \clIntVal{\clI}{0}{}$, $\clIi = \clIntVal{\clI}{}{1}$, and $\clIc = \clIntVal{\clI}{0}{1}$.)

The clones on $\{0,1\}$ are well known; they were described by Post~\cite{Post}.
The countably infinite lattice of clones of Boolean functions, also known as \emph{Post's lattice,} is presented in Figure~\ref{fig:Post}.
The notation for these clones was defined above.
However, it should be noted that not all of the classes defined above are clones.

\begin{figure}
\begin{center}
\scalebox{0.31}{%
\tikzstyle{every node}=[circle, draw, fill=black, scale=1, font=\LARGE]
\begin{tikzpicture}[baseline, scale=1]
   \node [label = below:$\clIc$] (Ic) at (0,-1) {};
   \node [label = right:$\clIstar$] (Istar) at (0,0.5) {};
   \node [label = right:$\clIo$] (I0) at (4.5,0.5) {};
   \node [label = left:$\clIi$] (I1) at (-4.5,0.5) {};
   \node [label = below:$\clI$] (I) at (0,2) {};
   \node [label = above:$\clOmegaOne$] (Omega1) at (0,5) {};
   \node [label = below:$\clLc$] (Lc) at (0,7.5) {};
   \node [label = right:$\clLS$] (LS) at (0,9) {};
   \node [label = right:$\clLo$] (L0) at (3,9) {};
   \node [label = left:$\clLi$] (L1) at (-3,9) {};
   \node [label = above:$\clL$] (L) at (0,10.5) {};
   \node [label = below:$\clSM\,\,$] (SM) at (0,13.5) {};
   \node [label = left:$\clSc$] (Sc) at (0,15) {};
   \node [label = above:$\clS$] (S) at (0,16.5) {};
   \node [label = below:$\clMc$] (Mc) at (0,23) {};
   \node [label = left:$\clMo\,\,$] (M0) at (2,24) {};
   \node [label = right:$\,\,\clMi$] (M1) at (-2,24) {};
   \node [label = above:$\clM\,\,$] (M) at (0,25) {};
   \node [label = below:$\clLambdac$] (Lamc) at (7.2,6.7) {};
   \node [label = left:$\clLambdai$] (Lam1) at (5,7.5) {};
   \node [label = right:$\clLambdao$] (Lam0) at (8.7,7.5) {};
   \node [label = below:$\clLambda$] (Lam) at (6.5,8.3) {};
   \node [label = left:$\clMcUk{\infty}$] (McUi) at (7.2,11.5) {};
   \node [label = left:$\clMUk{\infty}$] (MUi) at (8.7,13) {};
   \node [label = right:$\clTcUk{\infty}$] (TcUi) at (10.2,12) {};
   \node [label = right:$\clUk{\infty}$] (Ui) at (11.7,13.5) {};
   \node [label = left:$\clMcUk{3}$] (McU3) at (7.2,16) {};
   \node [label = left:$\clMUk{3}$] (MU3) at (8.7,17.5) {};
   \node [label = right:$\clTcUk{3}$] (TcU3) at (10.2,16.5) {};
   \node [label = right:$\clUk{3}$] (U3) at (11.7,18) {};
   \node [label = left:$\clMcUk{2}\,$] (McU2) at (7.2,19) {};
   \node [label = left:$\clMUk{2}\,$] (MU2) at (8.7,20.5) {};
   \node [label = right:$\clTcUk{2}$] (TcU2) at (10.2,19.5) {};
   \node [label = right:$\clUk{2}$] (U2) at (11.7,21) {};
   \node [label = below:$\clVc$] (Vc) at (-7.2,6.7) {};
   \node [label = right:$\clVo$] (V0) at (-5,7.5) {};
   \node [label = left:$\clVi$] (V1) at (-8.7,7.5) {};
   \node [label = below:$\clV$] (V) at (-6.5,8.3) {};
   \node [label = right:$\clMcWk{\infty}$] (McWi) at (-7.2,11.5) {};
   \node [label = right:$\clMWk{\infty}$] (MWi) at (-8.7,13) {};
   \node [label = left:$\clTcWk{\infty}$] (TcWi) at (-10.2,12) {};
   \node [label = left:$\clWk{\infty}$] (Wi) at (-11.7,13.5) {};
   \node [label = right:$\clMcWk{3}$] (McW3) at (-7.2,16) {};
   \node [label = right:$\clMWk{3}$] (MW3) at (-8.7,17.5) {};
   \node [label = left:$\clTcWk{3}$] (TcW3) at (-10.2,16.5) {};
   \node [label = left:$\clWk{3}$] (W3) at (-11.7,18) {};
   \node [label = right:$\,\,\clMcWk{2}$] (McW2) at (-7.2,19) {};
   \node [label = right:$\clMWk{2}$] (MW2) at (-8.7,20.5) {};
   \node [label = left:$\clTcWk{2}$] (TcW2) at (-10.2,19.5) {};
   \node [label = left:$\clWk{2}$] (W2) at (-11.7,21) {};
   \node [label = above:$\clOI$] (Tc) at (0,28) {};
   \node [label = right:$\clOX$] (T0) at (5,29.5) {};
   \node [label = left:$\clXI$] (T1) at (-5,29.5) {};
   \node [label = above:$\clAll$] (Omega) at (0,31) {};
   \draw [thick] (Ic) -- (Istar) to[out=135,in=-135] (Omega1);
   \draw [thick] (I) -- (Omega1);
   \draw [thick] (Omega1) to[out=135,in=-135] (L);
   \draw [thick] (Ic) -- (I0) -- (I);
   \draw [thick] (Ic) -- (I1) -- (I);
   \draw [thick] (Ic) to[out=128,in=-134] (Lc);
   \draw [thick] (Ic) to[out=58,in=-58] (SM);
   \draw [thick] (I0) -- (L0);
   \draw [thick] (I1) -- (L1);
   \draw [thick] (Istar) to[out=60,in=-60] (LS);
   \draw [thick] (Ic) -- (Lamc);
   \draw [thick] (I0) -- (Lam0);
   \draw [thick] (I1) -- (Lam1);
   \draw [thick] (I) -- (Lam);
   \draw [thick] (Ic) -- (Vc);
   \draw [thick] (I0) -- (V0);
   \draw [thick] (I1) -- (V1);
   \draw [thick] (I) -- (V);
   \draw [thick] (Lamc) -- (Lam0) -- (Lam);
   \draw [thick] (Lamc) -- (Lam1) -- (Lam);
   \draw [thick] (Lamc) -- (McUi);
   \draw [thick] (Lam0) -- (MUi);
   \draw [thick] (Lam1) -- (M1);
   \draw [thick] (Lam) -- (M);
   \draw [thick] (Vc) -- (V0) -- (V);
   \draw [thick] (Vc) -- (V1) -- (V);
   \draw [thick] (Vc) -- (McWi);
   \draw [thick] (V0) -- (M0);
   \draw [thick] (V1) -- (MWi);
   \draw [thick] (V) -- (M);
   \draw [thick] (McUi) -- (TcUi) -- (Ui);
   \draw [thick] (McUi) -- (MUi) -- (Ui);
   \draw [thick,loosely dashed] (McUi) -- (McU3);
   \draw [thick,loosely dashed] (MUi) -- (MU3);
   \draw [thick,loosely dashed] (TcUi) -- (TcU3);
   \draw [thick,loosely dashed] (Ui) -- (U3);
   \draw [thick] (McU3) -- (TcU3) -- (U3);
   \draw [thick] (McU3) -- (MU3) -- (U3);
   \draw [thick] (McU3) -- (McU2);
   \draw [thick] (MU3) -- (MU2);
   \draw [thick] (TcU3) -- (TcU2);
   \draw [thick] (U3) -- (U2);
   \draw [thick] (McU2) -- (TcU2) -- (U2);
   \draw [thick] (McU2) -- (MU2) -- (U2);
   \draw [thick] (McU2) -- (Mc);
   \draw [thick] (MU2) -- (M0);
   \draw [thick] (TcU2) to[out=120,in=-25] (Tc);
   \draw [thick] (U2) -- (T0);
   \draw [thick] (McWi) -- (TcWi) -- (Wi);
   \draw [thick] (McWi) -- (MWi) -- (Wi);
   \draw [thick,loosely dashed] (McWi) -- (McW3);
   \draw [thick,loosely dashed] (MWi) -- (MW3);
   \draw [thick,loosely dashed] (TcWi) -- (TcW3);
   \draw [thick,loosely dashed] (Wi) -- (W3);
   \draw [thick] (McW3) -- (TcW3) -- (W3);
   \draw [thick] (McW3) -- (MW3) -- (W3);
   \draw [thick] (McW3) -- (McW2);
   \draw [thick] (MW3) -- (MW2);
   \draw [thick] (TcW3) -- (TcW2);
   \draw [thick] (W3) -- (W2);
   \draw [thick] (McW2) -- (TcW2) -- (W2);
   \draw [thick] (McW2) -- (MW2) -- (W2);
   \draw [thick] (McW2) -- (Mc);
   \draw [thick] (MW2) -- (M1);
   \draw [thick] (TcW2) to[out=60,in=-155] (Tc);
   \draw [thick] (W2) -- (T1);
   \draw [thick] (SM) -- (McU2);
   \draw [thick] (SM) -- (McW2);
   \draw [thick] (Lc) -- (LS) -- (L);
   \draw [thick] (Lc) -- (L0) -- (L);
   \draw [thick] (Lc) -- (L1) -- (L);
   \draw [thick] (Lc) to[out=120,in=-120] (Sc);
   \draw [thick] (LS) to[out=60,in=-60] (S);
   \draw [thick] (L0) -- (T0);
   \draw [thick] (L1) -- (T1);
   \draw [thick] (L) to[out=125,in=-125] (Omega);
   \draw [thick] (SM) -- (Sc) -- (S);
   \draw [thick] (Sc) to[out=142,in=-134] (Tc);
   \draw [thick] (S) to[out=42,in=-42] (Omega);
   \draw [thick] (Mc) -- (M0) -- (M);
   \draw [thick] (Mc) -- (M1) -- (M);
   \draw [thick] (Mc) to[out=120,in=-120] (Tc);
   \draw [thick] (M0) -- (T0);
   \draw [thick] (M1) -- (T1);
   \draw [thick] (M) to[out=55,in=-55] (Omega);
   \draw [thick] (Tc) -- (T0) -- (Omega);
   \draw [thick] (Tc) -- (T1) -- (Omega);
\end{tikzpicture}
}
\end{center}
\caption{Post's lattice.}
\label{fig:Post}
\end{figure}
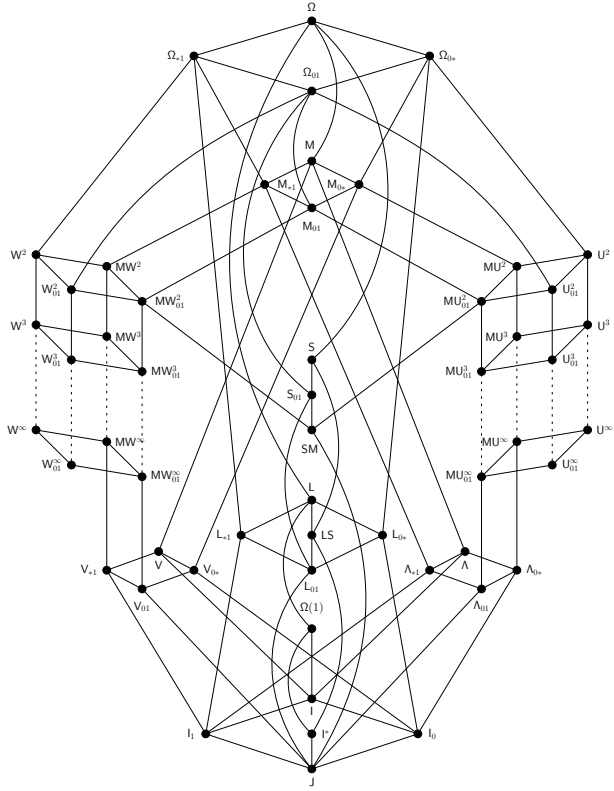

As a notational tool,
for sets $\mathcal{K}$ and $\mathcal{K}'$ of clones on $\{0,1\}$, we denote by $[\mathcal{K}, \mathcal{K}']$ the interval
\[
\{ \, C \mid \text{$C$ is a clone on $\{0,1\}$ such that $K \subseteq C \subseteq K'$ for some $K \in \mathcal{K}$ and $K' \in \mathcal{K}'$} \, \}
\]
in Post's lattice.
We usually simplify this notation when one of the sets $\mathcal{K}$ and $\mathcal{K}'$ is a singleton by dropping set brackets;
thus, for example, we may write
$[\clLc,\clL]$ and $[\{\clSM, \clMcUk{\ell}, \clMcWk{\ell} \},\clAll]$
in place of 
$[\{\clLc\}, \{\clL\}]$ and $[\{\clSM, \clMcUk{\ell}, \clMcWk{\ell} \}, \{\clAll\}]$.


\section{Helpful facts about clonoids}
\label{sec:helpful}

For the needs of subsequent sections, we begin by recalling some auxiliary results and proving a few new ones about $(C_1,C_2)$\hyp{}clonoids and the cardinality of the lattice $\closys{(C_1,C_2)}$ of $(C_1,C_2)$\hyp{}clonoids.

For $c \in B$, let $\clVakaAB{c}{AB}$ be the set of all constant functions in $\mathcal{F}_{AB}$ taking value $c$.
For a subset $S \subseteq B$, let $\clVakaAB{S}{AB} = \bigcup_{c \in S} \clVakaAB{c}{AB}$.
If $A = B$, we write simply $\clVakaAB{c}{A}$ and $\clVakaAB{S}{A}$ for $\clVakaAB{c}{AA}$ and $\clVakaAB{S}{AA}$, respectively,
or we may omit the superscripts if the sets $A$ and $B$ are clear from the context.

\begin{lemma}[{\cite[Lemma~2.10]{Lehtonen-discmono}}]
\label{lem:C1C2Vak}
Let $C_1$ and $C_2$ be clones on $A$ and $B$, respectively, and let $S \subseteq B$.
Assume that $C_2 \cup \clVakaAB{S}{B}$ is a clone on $B$.
\begin{enumerate}[label=\upshape{(\roman*)}]
\item\label{lem:C1C2Vak:i}
If $F \subseteq \mathcal{F}_{AB}$ is a $(C_1,C_2)$\hyp{}clonoid,
then
$F \cup \clVakaAB{S}{AB}$ is a $(C_1, C_2 \cup \clVakaAB{S}{B})$\hyp{}clonoid.
\item\label{lem:C1C2Vak:ii}
The nonempty $(C_1, C_2 \cup \clVakaAB{S}{B})$\hyp{}clonoids are precisely the $(C_1,C_2)$\hyp{}clonoids $K$ satisfying $\clVakaAB{S}{AB} \subseteq K$.
\end{enumerate}
\end{lemma}

\begin{proposition}
\label{prop:C2-const}
Let $C_1$ and $C_2$ be clones on $A$ and $B$, respectively, let $S \subseteq B$, and assume that $C_2 \cup \clVakaAB{S}{B}$ is a clone on $B$.
Then the following statements hold.
\begin{enumerate}[label=\upshape{(\roman*)}]
\item $\closys{(C_1, C_2 \cup \clVakaAB{S}{B})}$ is finite if and only if $\closys{(C_1,C_2)}$ is finite.
\item $\closys{(C_1, C_2 \cup \clVakaAB{S}{B})}$ is countably infinite if and only if $\closys{(C_1,C_2)}$ is countably infinite.
\item $\closys{(C_1, C_2 \cup \clVakaAB{S}{B})}$ is uncountable if and only if $\closys{(C_1,C_2)}$ is uncountable.
\end{enumerate}
\end{proposition}

\begin{proof}
The three statements will follow by exhaustion if we show that $\closys{(C_1, C_2 \cup \clVakaAB{S}{B})}$ is finite (countably infinite, uncountably infinite, resp.)\ whenever $\closys{(C_1,C_2)}$ is so.

By the monotonicity of function class composition, we have $\card{\closys{(C_1, C_2 \cup \clVakaAB{S}{B})}} \leq \card{\closys{(C_1,C_2)}}$.
From this it follows immediately that $\closys{(C_1, C_2 \cup \clVakaAB{S}{B})}$ is finite whenever $\closys{(C_1,C_2)}$ is finite.

For each $T \subseteq S$, let
\[
\mathcal{A}_T := \{ \, F \in \closys{(C_1,C_2)} \mid F \cap \clVakaAB{S}{AB} = \clVakaAB{T}{AB} \, \}.
\]
Each $\mathcal{A}_T$ is a fragment of $\closys{(C_1,C_2)}$; the nonempty sets of this form constitute a partition of $\closys{(C_1,C_2)}$.

Assume now that $\closys{(C_1,C_2)}$ is (countably, uncountably) infinite.
Since there are only a finite number of subsets of $S$, at least one of the fragments $\mathcal{A}_T$ is (countably, uncountably) infinite, say $\mathcal{A}_D$.
Now, the map $F \mapsto F \cup \clVakaAB{S}{AB}$ is an injection from $\mathcal{A}_D$ to $\mathcal{A}_S$; therefore $\card{\mathcal{A}_D} \leq \card{\mathcal{A}_S}$.
By Lemma~\ref{lem:C1C2Vak}, every member of $\mathcal{A}_S$ is a $(C_1, C_2 \cup \clVakaAB{S}{B})$\hyp{}clonoid.
Therefore $\card{\mathcal{A}_S} \leq \card{\closys{(C_1, C_2 \cup \clVakaAB{S}{B})}} \leq \card{\closys{(C_1,C_2)}} = \card{\mathcal{A}_D} \leq \card{\mathcal{A}_S}$.
Consequently, $\closys{(C_1, C_2 \cup \clVakaAB{S}{B})}$ is (countably, uncountably) infinite whenever $\closys{(C_1,C_2)}$ is so.
\end{proof}

\begin{corollary}
\label{cor:Bfintervals}
Let $C_1$, $C_2$, and $C'_2$ be clones on $\{0,1\}$.
If both $C_2$ and $C'_2$ belong to one of the intervals $[\clIc,\clI]$, $[\clIstar,\clOmegaOne]$, $[\clVc,\clV]$, $[\clLambdac,\clLambda]$, $[\clMcUk{k},\clMUk{k}]$, $[\clMcWk{k},\clMWk{k}]$, for $k \in \{2, 3, \dots, \infty\}$, $[\clMc,\clM]$,
then both $\closys{(C_1,C_2)}$ and $\closys{(C_1,C'_2)}$ are finite, both are countably infinite, or both are uncountable.
\end{corollary}

\begin{proof}
This follows from Proposition~\ref{prop:C2-const} by observing that
\begin{align*}
\clI &= \clIc \cup \clVak, &
\clOmegaOne &= \clIstar \cup \clVak, &
\clV &= \clVc \cup \clVak, &
\clLambda &= \clLambdac \cup \clVak, \\
\clMUk{k} &= \clMcUk{k} \cup \clVako, &
\clMWk{k} &= \clMcWk{k} \cup \clVaki, &
\clM &= \clMc \cup \clVak. &
&
\qedhere
\end{align*}
\end{proof}

\begin{lemma}
\label{lem:IJ-InJi}
For $I, J \subseteq \clAll$, we have $I J = I^\mathrm{n} \overline{J}$.
\end{lemma}

\begin{proof}
Let $f \in I J$.
Then $f = g ( h_1, \dots, h_n )$ for $g \in I$, $h_1, \dots, h_n \in J$.
Now,
\begin{align*}
f(\vect{a})
= g ( h_1(\vect{a}), \dots, h_n(\vect{a}) )
= g ( \overline{\overline{h_1(\vect{a})}}, \dots, \overline{\overline{h_n(\vect{a})}} )
= g^\mathrm{n} ( \overline{h_1}(\vect{a}), \dots, \overline{h_n}(\vect{a}) ),
\end{align*}
and therefore $f = g^\mathrm{n} ( \overline{h_1}, \dots, \overline{h_n} ) \in I^\mathrm{n} \overline{J}$.
This shows that $I J \subseteq I^\mathrm{n} \overline{J}$.
Because $(I^\mathrm{n})^\mathrm{n} = I$ and $\overline{\overline{J}} = J$, it also follows from the above that $I^\mathrm{n} \overline{J} \subseteq I J$.
\end{proof}

\begin{proposition}
\label{prop:Knid}
Let $C_1$ and $C_2$ be clones on $\{0,1\}$, and let $K \subseteq \clAll$.
Then the following statements are equivalent:
\begin{enumerate}[label=\upshape{(\alph*)}]
\item\label{prop:Knid:K} $K$ is a $(C_1,C_2)$\hyp{}clonoid.
\item\label{prop:Knid:Kn} $K^\mathrm{n}$ is a $(C_1^\mathrm{d},C_2)$\hyp{}clonoid.
\item\label{prop:Knid:Ki} $\overline{K}$ is a $(C_1,C_2^\mathrm{d})$\hyp{}clonoid.
\item\label{prop:Knid:Kd} $K^\mathrm{d}$ is a $(C_1^\mathrm{d},C_2^\mathrm{d})$\hyp{}clonoid.
\end{enumerate}
\end{proposition}

\begin{proof}
\ref{prop:Knid:K} $\implies$ \ref{prop:Knid:Kn}:
Assume that $K$ is a $(C_1,C_2)$\hyp{}clonoid.
This is equivalent to $K = \gen[(C_1,C_2)]{K} = C_2 ( K C_1 )$.
By Lemma~\ref{lem:IJ-InJi} we get
\begin{align*}
& K^\mathrm{n} = ( C_2 ( K C_1 ) )^\mathrm{n} = C_2 ( K C_1 )^\mathrm{n} = C_2 ( K C_1^\mathrm{n} ) = C_2 ( K^\mathrm{n} \overline{C_1^\mathrm{n}} ) = C_2 ( K^\mathrm{n} C_1^\mathrm{d} ), \\
& \overline{K} = \overline{ C_2 ( K C_1 ) } = \overline{C_2} ( K C_1 ) = (\overline{C_2})^\mathrm{n} \overline{ ( K C_1 ) } = C_2^\mathrm{d} \overline{ ( K C_1 ) } = C_2^\mathrm{d} ( \overline{K} C_1 ), \\
& K^\mathrm{d} = \overline{ K^\mathrm{n} } = \overline{  C_2 ( K^\mathrm{n} C_1^\mathrm{d} ) } = \overline{C_2} ( K^\mathrm{n} C_1^\mathrm{d} ) = (\overline{C_2})^\mathrm{n} \overline{( K^\mathrm{n} C_1^\mathrm{d} )} = C_2^\mathrm{d} ( \overline{K^\mathrm{n}} C_1^\mathrm{d} ) = C_2^\mathrm{d} ( K^\mathrm{d} C_1^\mathrm{d} ).
\end{align*}
This shows that statement \ref{prop:Knid:K} implies the others.
To show the remaining implications, we just consider $K^\mathrm{n}$, $\overline{K}$, or $K^\mathrm{d}$ in place of $K$ and $C_1^\mathrm{d}$ and $C_2^\mathrm{d}$ in place of $C_1$ and $C_2$, respectively.
\end{proof}

\begin{corollary}
\label{cor:duals}
Let $C_1$ and $C_2$ be clones on $\{0,1\}$.
The lattices $\closys{(C_1,C_2)}$, $\closys{(C_1,C_2^\mathrm{d})}$, $\closys{(C_1^\mathrm{d},C_2)}$, and $\closys{(C_1^\mathrm{d},C_2^\mathrm{d})}$ have the same cardinality.
\end{corollary}

\begin{proof}
This follows immediately from Proposition~\ref{prop:Knid}.
\end{proof}

\begin{lemma}[{\cite[Lemma~3.2]{Lehtonen-discmono}}]
\label{lem:DM-3.2}
Let $C$ be a clone on $\{0,1\}$, and let $K$ be a $(C,\clIc)$\hyp{}clonoid.
Then the following statements hold.
\begin{enumerate}[label=\upshape{(\alph*)}]
\item $\overline{K}$ is a $(C,\clIc)$\hyp{}clonoid.
\item $K \cup \clVako$ is a $(C,\clIo)$\hyp{}clonoid.
\item $K \cup \clVaki$ is a $(C,\clIi)$\hyp{}clonoid.
\item $K \cup \clVak$ is a $(C,\clI$\hyp{}clonoid.
\item $K \cup \overline{K}$ is a $(C,\clIstar)$\hyp{}clonoid.
\item $K \cup \overline{K} \cup \clVak$ is a $(C,\clOmegaOne)$\hyp{}clonoid.
\end{enumerate}
\end{lemma}

\begin{lemma}[{\cite[Lemma~3.3]{Lehtonen-discmono}}]
\label{lem:DM-3.3}
Let $C$ be a clone on $\{0,1\}$.
The $(C,\clIstar)$\hyp{}clonoids are precisely the $(C,\clIc)$\hyp{}clonoids $K$ satisfying $K = \overline{K}$.
\end{lemma}

\begin{lemma}
\label{lem:a-pres}
Let $a \in \{0,1\}$, let $K \subseteq \clIntVal{\clAll}{a}{}$, and let $C_1$ be a subclone of $\clOX$.
When $a = 0$, let $C_2$ be a subclone of $\clOX$; when $a = 1$, let $C_2$ be a subclone of $\clXI$.
Then the following statements hold.
\begin{enumerate}[label=\upshape{(\roman*)}]
\item\label{lem:a-pres:1} $K C_1 \subseteq \clIntVal{\clAll}{a}{}$.
\item\label{lem:a-pres:2} $\gen[(C_1,C_2)]{K} \subseteq \clIntVal{\clAll}{a}{}$.
\end{enumerate}
\end{lemma}

\begin{proof}
\ref{lem:a-pres:1}
Let $f \in K C_1$.
Then $f = g(h_1, \dots, h_n)$ for some $g \in K$ and $h_1, \dots, h_n \in C_1$.
We have
\[
f(\vect{0}) = g(h_1(\vect{0}), \dots, h_n(\vect{0})) = g(a, \dots, a) = a,
\]
so $f \in \clIntVal{\clAll}{a}{}$.

\ref{lem:a-pres:2}
Let $f \in \gen[(C_1,C_2)]{K} = C_2 ( K C_1 )$.
Then $f = \varphi(\gamma_1, \dots, \gamma_n)$ for some $\varphi \in C_1$ and $\gamma_1, \dots, \gamma_n \in K C_1$.
By \ref{lem:a-pres:1}, $\gamma_1, \dots, \gamma_n \in \clIntVal{\clAll}{a}{}$.
We have
\[
f(\vect{0}) = \varphi(\gamma_1(\vect{0}), \dots, \gamma_n(\vect{0})) = \varphi(a, \dots, a) = a,
\]
so $f \in \clIntVal{\clAll}{a}{}$.
\end{proof}

\begin{lemma}
\label{lem:Kax}
Let $C_1$ be a subclone of $\clOX$, and let $K$ be a $(C_1,\clIc)$\hyp{}clonoid.
Then $\clIntVal{K}{0}{} = K \cap \clOX$ and $\clIntVal{K}{1}{} = K \cap \clIX$ are $(C_1,\clIc)$\hyp{}clonoids.
\end{lemma}

\begin{proof}
Because $K$ is a $(C_1,\clIc)$\hyp{}clonoid, we have $K = \gen[(C_1,\clIc)]{K} = \clIc ( K C_1 ) = K C_1$.
We have
$\clIntVal{K}{0}{} C_1 \subseteq K C_1 = K$ by the monotonicity of function class composition and $\clIntVal{K}{0}{} C_1 \subseteq \clOX$ by Lemma~\ref{lem:a-pres}; therefore, $\clIntVal{K}{0}{} C_1 \subseteq K \cap \clOX = \clIntVal{K}{0}{}$.
Similarly,
$\clIntVal{K}{1}{} C_1 \subseteq K C_1 = K$ and $\clIntVal{K}{1}{} C_1 \subseteq \clIX$; therefore, $\clIntVal{K}{1}{} C_1 \subseteq K \cap \clIX = \clIntVal{K}{1}{}$.
\end{proof}

\begin{proposition}
\label{prop:card-Ic-Istar}
Let $C_1$ be a subclone of $\clOX$ or of $\clXI$.
\begin{enumerate}[label=\upshape{(\roman*)}]
\item\label{prop:card-Ic-Istar:U} $\closys{(C_1,\clIc)}$ is uncountable if and only if $\closys{(C_1,\clIstar)}$ is uncountable.
\item\label{prop:card-Ic-Istar:C} $\closys{(C_1,\clIc)}$ is countably infinite if and only if $\closys{(C_1,\clIstar)}$ is countably infinite.
\item\label{prop:card-Ic-Istar:F} $\closys{(C_1,\clIc)}$ is finite if and only if $\closys{(C_1,\clIstar)}$ is finite.
\end{enumerate}
\end{proposition}

\begin{proof}
We assume that $C_1$ is a subclone of $\clOX$. The claim about subclones of $\clXI$ follows by duality from Proposition~\ref{prop:Knid}.

Partition $\closys{(C_1,\clIc)}$ into four parts:
\begin{align*}
L_1 &:= \{ \, K \in \closys{(C_1,\clIc)} \mid \clIntVal{K}{0}{} \neq \clEmpty, \, \clIntVal{K}{1}{} \neq \clEmpty \, \}, \\
L_2 &:= \{ \, K \in \closys{(C_1,\clIc)} \mid \clIntVal{K}{0}{} \neq \clEmpty, \, \clIntVal{K}{1}{} = \clEmpty \, \}, \\
L_3 &:= \{ \, K \in \closys{(C_1,\clIc)} \mid \clIntVal{K}{0}{} = \clEmpty, \, \clIntVal{K}{1}{} \neq \clEmpty \, \}, \\
L_4 &:= \{ \, K \in \closys{(C_1,\clIc)} \mid \clIntVal{K}{0}{} = \clEmpty, \, \clIntVal{K}{1}{} = \clEmpty \, \}.
\end{align*}
Obviously, $L_4 = \{\clEmpty\}$.

It follows from Lemma~\ref{lem:DM-3.2} that $L_2 = \overline{L_3}$, and the map $K \mapsto \overline{K}$ is a bijection between $L_2$ and $L_3$; therefore, $\card{L_2} = \card{L_3}$.
Moreover, the map $K \mapsto K \cup \overline{K}$ is an injection from $L_2$ to $L_1$ (equivalently, from $L_3$ to $L_1$; therefore $\card{L_2} = \card{L_3} \leq \card{L_1}$.

Observe also that $\card{L_1} \leq \card{L_2 \times L_3}$, because the map $L_1 \to L_2 \times L_3$, $K \mapsto (\clIntVal{K}{0}{}, \clIntVal{K}{1}{})$ is clearly an injection; note that $\clIntVal{K}{0}{}$ and $\clIntVal{K}{1}{}$ are indeed $(C_1,\clIc)$\hyp{}clonoids by Lemma~\ref{lem:Kax}.
Thus, we have
\[
\card{L_2} = \card{L_3} \leq \card{L_1} \leq \card{L_2 \times L_3}.
\]
Because the Cartesian product of finite sets is finite, and the product of countable sets is countable,
it follows that the three parts $L_1$, $L_2$, and $L_3$ are either all finite, all three are countably infinite, or all three are uncountable.
Because a finite union of finite (countably infinite, uncountable, resp.)\ sets is finite (countably infinite, uncountable, resp.), it follows that if $\closys{(C_1,\clIc)}$ is finite (countably infinite, uncountable, resp.)\, then so are $L_1$, $L_2$, and $L_3$.

Because for each $(C_1,\clIc)$\hyp{}clonoid $K$, the class $K \cup \overline{K}$ is a $(C_1,\clIstar)$\hyp{}clonoid by Lemma~\ref{lem:DM-3.2}, it follows that $K \mapsto K \cup \overline{K}$ is an injection $L_2 \to \closys{(C_1,\clIstar)}$; therefore, $\card{L_2} \leq \card{\closys{(C_1,\clIstar)}}$.
Because $\closys{(C_1,\clIstar)} \subseteq \closys{(C_1,\clIc)}$, we have $\card{\closys{(C_1,\clIstar)}} \leq \card{\closys{(C_1,\clIc)}}$.

Now, for statement \ref{prop:card-Ic-Istar:U}, the implication ``$\Leftarrow$'' is clear, and ``$\Rightarrow$'' follows from the above observations (if $\closys{(C_1,\clIc)}$ is uncountable, then so is $L_2$ and hence also $\closys{(C_1,\clIstar)}$).

For statement \ref{prop:card-Ic-Istar:C}, if $\closys{(C_1,\clIstar)}$ is countably infinite, then $\closys{(C_1,\clIc)}$ is at least countably infinite; but $\closys{(C_1,\clIstar)}$ cannot be uncountable by \ref{prop:card-Ic-Istar:U}. Conversely, if $\closys{(C_1,\clIc)}$ is countably infinite, then so is $L_2$, and therefore $\closys{(C_1,\clIstar)}$ is at least countably infinite; $\closys{(C_1,\clIstar)}$ cannot be uncountable by \ref{prop:card-Ic-Istar:U}.

For statement \ref{prop:card-Ic-Istar:F}, the implication ``$\Rightarrow$'' is clear, and the implication ``$\Leftarrow$'' follows by contraposition from \ref{prop:card-Ic-Istar:U} and \ref{prop:card-Ic-Istar:C}.
\end{proof}


\section{Pippenger's functions and uncountable clonoid lattices}
\label{sec:uncountable}

The main result of this section is the following theorem that shows that for certain pairs $(C_1,C_2)$ of clones, the lattice $\closys{(C_1,C_2)}$ of $(C_1,C_2)$\hyp{}clonoids is uncountably infinite.

\begin{theorem}
\label{thm:uncountable}
For clones $C_1$ and $C_2$ such that $C_1 \subseteq K_1$ and $C_2 \subseteq K_2$ for some
\[
\begin{split}
(K_1, K_2) \in
\{
&
(\clOmegaOne, \clOmegaOne),
(\clOmegaOne, \clLambda),
(\clOmegaOne, \clV),
(\clIstar, \clUinf),
(\clIstar, \clWinf), \\
&
(\clIo, \clUinf),
(\clIi, \clUinf),
(\clIo, \clWinf),
(\clIi, \clWinf),
(\clL, \clOmegaOne), \\
&
(\clLambda, \clLambda),
(\clLambda, \clV),
(\clV, \clLambda),
(\clV, \clV),
(\clLambda, \clOmegaOne),
(\clV, \clOmegaOne)
\}
\end{split}
\]
there are an uncountable infinitude of $(C_1,C_2)$\hyp{}clonoids.
\end{theorem}

The remainder of this section is devoted to the proof of this result.
The proof is based on exhibiting, for each pair $(C_1,C_2)$ of clones as prescribed in the theorem, a countably infinite family of Boolean functions with the property that distinct subsets of this family always generate distinct $(C_1,C_2)$\hyp{}clonoids.
Because the power set of a countably infinite set is uncountable, it follows that there are an uncountable infinitude of $(C_1,C_2)$\hyp{}clonoids.

We will make use of two such countably infinite families of functions.
One of these families (see Definition~\ref{def:fn}) has proved useful in proving similar results about Boolean functions.
In fact, Pippenger \cite[Proposition~3.4]{Pippenger} used exactly this family of functions and the same type of argument to show that there are an uncountable infinitude of $(\clOmegaOne,\clIc)$\hyp{}clonoids of Boolean functions.

\begin{definition}
For $S \subseteq \nset{n}$, we denote by $\vect{e}_S$ the \emph{characteristic $n$\hyp{}tuple} of $S$, i.e., the tuple $(a_1, \dots, a_n) \in \{0,1\}^n$ satisfying $a_i = 1$ if and only if $i \in S$.
We write $\vect{e}_i$ for $\vect{e}_{\{i\}}$.
(The arity $n$ is implicit in the notation but will be clear from the context.)
\end{definition}

\begin{definition}
Let $\vect{a}, \vect{b} \in \{0,1\}^n$.
The \emph{Hamming distance} betwen $\vect{a}$ and $\vect{b}$ is $d(\vect{a},\vect{b}) := \{ \, i \in \nset{n} \mid a_i \neq b_i \,\}$, the number of positions in which the two tuples are different.
The \emph{Hamming weight} of $\vect{a}$ is $w(\vect{a}) := \{ \, i \in \nset{n} \mid a_i \neq 0 \, \}$, the number of non\hyp{}zero entries of $\vect{a}$.
We clearly have $w(\vect{a}) = d(\vect{a}, \vect{0})$ and $d(\vect{a}, \vect{b}) = w(\vect{a} + \vect{b})$.
\end{definition}

\begin{definition}[{Pippenger~\cite[Proposition~3.4]{Pippenger}}]
\label{def:fn}
For $n \in \IN$ with $n \geq 3$, we define $f_n \colon \{0,1\}^n \to \{0,1\}$ by the rule
$f_n(\vect{a}) = 1$ if and only if $w(\vect{a}) \in \{1, n-1\}$.
In other words, the true points of $f_n$ are the $n$\hyp{}tuples $\vect{e}_i$ and $\overline{\vect{e}_i}$ for $i \in \nset{n}$.
For $S \subseteq \IN^{+} \setminus \nset{2}$, let $F_S := \{ \, f_n \mid n \in S \, \}$.
\end{definition}

\begin{definition}
\label{def:qn}
For $n \in \IN$ with $n \geq 3$, we define $q_n \colon \{0,1\}^n \to \{0,1\}$ by the rule
$q_n(\vect{a}) = 1$ if and only if $w(\vect{a}) \in \{1, n\}$.
In other words, the true points of $q_n$ are the $n$\hyp{}tuples $\vect{1}$ and $\vect{e}_i$ for $i \in \nset{n}$.
For $S \subseteq \IN^{+} \setminus \nset{2}$, let $Q_S := \{ \, q_n \mid n \in S \, \}$.
\end{definition}

\begin{observation}
\label{obs:distances}
The Hamming distances between true points of $f_n$ are, for $i, j \in \nset{n}$ with $i \neq j$,
\begin{align*}
& d(\vect{e}_i, \vect{e}_i) = w(\vect{0}) = 0, &
& d(\vect{e}_i, \vect{e}_j) = w(\vect{e}_{\{i,j\}}) = 2, \\
& d(\vect{e}_i, \overline{\vect{e}_i}) = w(\vect{1}) = n, &
& d(\vect{e}_i, \overline{\vect{e}_j}) = w(\vect{e}_{\nset{n} \setminus \{i,j\}}) = n-2, \\
& d(\overline{\vect{e}_i}, \overline{\vect{e}_i}) = w(\vect{0}) = 0, &
& d(\overline{\vect{e}_i}, \overline{\vect{e}_j}) = w(\vect{e}_{\{i,j\}}) = 2.
\end{align*}
\end{observation}

\begin{observation}
\label{obs:first3}
Let $n \geq 4$.
For every $\vect{u} \in \{0,1\}^3$, there exists a $\vect{v} \in \{0,1\}^{n-3}$ such that $\vect{u} \vect{v} \in f_n^{-1}(1)$ ($\vect{u} \vect{v}$ denotes the concatenation of $\vect{u}$ and $\vect{v}$), namely,
\begin{align*}
\vect{e}_4 &= \underline{000}10 \dots 0, &
\vect{e}_1 &= \underline{100}00 \dots 0, &
\vect{e}_2 &= \underline{010}00 \dots 0, &
\vect{e}_3 &= \underline{001}00 \dots 0, \\
\overline{\vect{e}_4} &= \underline{111}01 \dots 1, &
\overline{\vect{e}_1} &= \underline{011}11 \dots 1, &
\overline{\vect{e}_2} &= \underline{101}11 \dots 1, &
\overline{\vect{e}_3} &= \underline{110}11 \dots 1.
\end{align*}
Moreover, by negating the first component of $\vect{v}$, we obtain a tuple $\vect{v}' \in \{0,1\}^{n-3}$ such that $\vect{u} \vect{v}' \in f_n^{-1}(0)$.
Consequently, for every $\vect{a} \in \{0,1\}^n$, there exist $\vect{e} \in f_n^{-1}(1)$ and $\vect{e}' \in f_n^{-1}(0)$ such that the first three components of $\vect{a}$, $\vect{e}$, and $\vect{e}'$ coincide.
\end{observation}

\begin{definition}
Let $f, g \in \{0,1\}^n$.
We say that $f$ is a \emph{minorant} of $g$ and that $g$ is a \emph{majorant} of $f$, and we write $f \minorant g$,
if $f(\vect{a}) \leq g(\vect{a})$ for all $\vect{a} \in \{0,1\}^n$.
\end{definition}

\begin{lemma}
\label{lem:fn-Omega1}
Let $m, n \geq 5$.
\begin{enumerate}[label=\upshape{(\roman*)}]
\item\label{lem:fn-Omega1:Omega1}
If $\varphi \in \{ f_m \} \, \clOmegaOne$ and $f_n \minorant \varphi$ or $\neg \varphi \minorant f_n$, then $\varphi = \vak{1}$ or $m = n$.

\item\label{lem:fn-Omega1:IoIi}
If $\varphi \in \{ f_m \} \, \clIo$ or $\varphi \in \{ f_m \} \, \clIi$ and $f_n \minorant \varphi$, then $m = n$.
\end{enumerate}
\end{lemma}

\begin{proof}
\ref{lem:fn-Omega1:Omega1}
Assume $\varphi \in \{ f_m \} \, \clOmegaOne$.
Then $\varphi = f_m(g_1, \dots, g_m)$ for some $g_1, \dots, g_m \in \clOmegaOne$.
Taking $g \colon \{0,1\}^n \to \{0,1\}^m$, $g(\vect{a}) = (g_1(\vect{a}), \dots, g_m(\vect{a}))$, we have $\varphi = f_m \circ g$.
We can write $g_i = \gamma_{\sigma(i)} + \vak{d_i}$ for some $\sigma \colon \nset{m} \to \nset{n} \cup \{0\}$,
where $\gamma_i = \pr^{(n)}_i$ if $i \in \nset{n}$, $\gamma_0 = \vak{0}$, and $d_i \in \{0,1\}$.
Thus $g(\vect{a}) = \gamma(\vect{a}) + \vect{d}$, where $\gamma \colon \{0,1\}^n \to \{0,1\}^m$ is $\gamma(\vect{a}) = (\gamma_1(\vect{a}), \dots, \gamma_m(\vect{a}))$, and $\vect{d} = (d_1, \dots, d_m)$.

If $f_n \minorant \varphi$, then for every $\vect{e} \in f_n^{-1}(1)$, it holds that $f_m(g(\vect{e})) = 1$, i.e., $g(\vect{e}) \in f_m^{-1}(1)$.
In other words, $g$ maps true points of $f_n$ to true points of $f_m$.

If $\neg \varphi \minorant f_n$, then for every $\vect{e} \in f_n^{-1}(0)$, it holds that $\neg(\varphi(\vect{e})) = \neg(f_m(g(\vect{e}))) = 0$, i.e., $g(\vect{e}) \in f_m^{-1}(1)$.
In other words, $g$ maps false points of $f_n$ to true points of $f_m$.

\begin{claim}\label{clm:fn-Omega1:1}
Either $\varphi = \vak{1}$ or there are at least four elements $p \in \nset{n}$ such that $\sigma^{-1}(p) \neq \emptyset$.
\end{claim}

\begin{pfclaim}[Proof of Claim~\ref{clm:fn-Omega1:1}]
Assume there are fewer than four elements $p \in \nset{n}$ such that $\sigma^{-1}(p) \neq \emptyset$.
We may assume, without loss of generality, that these elements are $1, \dots, \ell$ for some $\ell \leq 3$.
Let $\vect{a} \in \{0,1\}^n$.
By Observation~\ref{obs:first3}, there exist $\vect{e} \in f_n^{-1}(1)$ and $\vect{e}' \in f_n^{-1}(0)$ such that the first $\ell$ components of $\vect{a}$, $\vect{e}$, and $\vect{e}'$ coincide.
Because $g$ does not depend on the arguments with indices greater than $\ell$, we have
$g(\vect{a}) = g(\vect{e}) = g(\vect{e}')$.
If $g$ maps true points of $f_n$ to true points of $f_m$, then
$\varphi(\vect{a}) = f_m(g(\vect{a})) = f_m(g(\vect{e})) = 1$.
If $g$ maps false points of $f_n$ to true points of $f_m$, then
$\varphi(\vect{a}) = f_m(g(\vect{a})) = f_m(g(\vect{e}')) = 1$.
In either case, we conclude that $\varphi = \vak{1}$.
\end{pfclaim}

In view of Claim~\ref{clm:fn-Omega1:1}, we assume from now on that there are at least four elements $p \in \nset{n}$ such that $\sigma^{-1}(p) \neq \emptyset$, i.e., $\card{\range \sigma \cap \nset{n}} \geq 4$.

\begin{claim}\label{clm:fn-Omega1:2}
$\range \sigma \cap \nset{n} = \nset{n}$.
\end{claim}

\begin{pfclaim}[Proof of Claim~\ref{clm:fn-Omega1:2}]
Suppose, to the contrary, that $\range \sigma \cap \nset{n} \subsetneq \nset{n}$, and let $p \in \nset{n} \setminus \range \sigma$.
Because $g$ does not depend on the $p$\hyp{}th argument, we have $g(\vect{e}_p) = g(\vect{0}) = \gamma(\vect{0}) + \vect{d} = \vect{0} + \vect{d} = \vect{d}$.
Since $\vect{0} \in f_n^{-1}(0)$ and $\vect{e}_p \in f_n^{-1}(1)$ and $g$ maps true points of $f_n$ to true points of $f_m$ or $g$ maps false points of $f_n$ to true points of $f_m$, it follows that $\vect{d} \in f_m^{-1}(1)$, so $\vect{d}$ equals $\vect{e}_q$ or $\overline{\vect{e}_q}$ for some $q \in \nset{m}$.

Now, for every $i \in \range \sigma \cap \nset{n}$ we have
$g(\vect{e}_i) = g(\vect{e}_{\{i,p\}}) = \gamma(\vect{e}_i) + \vect{d} = \vect{e}_{\sigma^{-1}(i)} + \vect{d}$.
Because $\vect{e}_i \in f_n^{-1}(1)$ and $\vect{e}_{\{i,p\}} \in f_n^{-1}(0)$ and $g(f_n^{-1}(1)) \subseteq f_m^{-1}(1)$ or $g(f_n^{-1}(0)) \subseteq f_m^{-1}(1)$, it follows that $\vect{e}_{\sigma^{-1}(i)} + \vect{d} \in f_m^{-1}(1)$, and therefore $\sigma^{-1}(i)$ equals either $\{q, k\}$ or $\nset{m} \setminus \{q, k\}$ for some $k \in \nset{m} \setminus \{q\}$.
Since $\sigma$\hyp{}preimages of distinct elements are disjoint, it follows that $\range \sigma \cap \nset{n}$ can have at most two elements.
This contradicts our assumption that $\card{\range \sigma \cap \nset{n}} \geq 4$.
\end{pfclaim}

\begin{claim}\label{clm:fn-Omega1:3}
$\card{\sigma^{-1}(i)} = 1$ for all $i \in \nset{n}$.
\end{claim}

\begin{pfclaim}[Proof of Claim~\ref{clm:fn-Omega1:3}]
Note that, by Claim~\ref{clm:fn-Omega1:2}, every element of $\nset{n}$ has a preimage under $\sigma$, i.e., $\card{\sigma^{-1}(i)} \geq 1$ for all $i \in \nset{n}$.
Suppose, to the contrary, that there is a $p \in \nset{n}$ such that $\card{\sigma^{-1}(p)} \geq 2$.
Let $q \in \nset{n}$ with $p \neq q$.
If $g$ maps true points of $f_n$ to true points of $f_m$, then $g(\vect{e}_p) = \vect{e}_{\sigma^{-1}(p)} + \vect{d}$ and $g(\vect{e}_q) = \vect{e}_{\sigma^{-1}(q)} + \vect{d}$ are in $f_m^{-1}(1)$, and $d(\vect{e}_{\sigma^{-1}(p)} + \vect{d}, \vect{e}_{\sigma^{-1}(q)} + \vect{d}) = w(\vect{e}_{\sigma^{-1}(\{p,q\})}) = \card{\sigma^{-1}(\{p,q\})}$.
If $g$ maps false points of $f_n$ to true points of $f_m$, then $g(\vect{0}) = \vect{d}$ and $g(\vect{e}_{\{p,q\}}) = \vect{e}_{\sigma^{-1}(\{p,q\}} + \vect{d}$ are in $f_m^{-1}(1)$, and $d(\vect{d}, \vect{e}_{\sigma^{-1}(\{p,q\})} + \vect{d}) = w(\vect{e}_{\sigma^{-1}(\{p,q\})}) = \card{\sigma^{-1}(\{p,q\})}$.
It follows from Observation~\ref{obs:distances} that $\card{\sigma^{-1}(\{p,q\})} \in \{0, 2, m-2, m\}$.
But cardinalities $0$ and $2$ are impossible, because $\card{\sigma^{-1}(\{p,q\})} > \card{\sigma^{-1}(p)} \geq 2$.
Cardinalities $m-2$ and $m$ are impossible, because if either of these were the case, then we would have
\begin{align*}
m & \geq
\card{\sigma^{-1}(\nset{n})}
= \card{\sigma^{-1}(\{p,q\})} + \card{\sigma^{-1}(\nset{n} \setminus \{p,q\})}
\\ &
\geq (m-2) + (n-2)
\geq (m-2) + 3
\geq m + 1,
\end{align*}
a contradiction.
(Note that we used the assumption $n \geq 5$ for the second last inequality above.)
\end{pfclaim}

It follows from Claims~\ref{clm:fn-Omega1:2} and \ref{clm:fn-Omega1:3} that for $R := \sigma^{-1}(\nset{n})$, $\sigma|_R \colon R \to \nset{n}$ is a bijection, and hence $m \geq n$.
Without loss of generality, we may assume that $\sigma(i) = i$ for all $i \in \nset{n}$ and $\sigma(j) = 0$ for $j \in \nset{m} \setminus \nset{n}$.
Thus, $g(\vect{a}) = \vect{a} \vect{0} + \vect{d}$, where $\vect{0}$ is an $(m-n)$\hyp{}tuple of $0$'s.
Observe that for all $\vect{a}, \vect{b} \in \{0,1\}^n$,
\begin{equation}
\begin{split}
d( g(\vect{a}), g(\vect{b}) )
&
= w( g(\vect{a}) + g(\vect{b}) )
= w( \vect{a} \vect{0} + \vect{d} + \vect{b} \vect{0} + \vect{d} )
\\ &
= w( \vect{a} \vect{0} + \vect{b} \vect{0} )
= w( \vect{a} + \vect{b} )
= d( \vect{a}, \vect{b} ).
\end{split}
\label{eq:dw}
\end{equation}

\begin{claim}\label{clm:fn-Omega1:4}
$\sigma^{-1}(0) = \emptyset$, i.e., $m = n$.
\end{claim}

\begin{pfclaim}[Proof of Claim~\ref{clm:fn-Omega1:4}]
Suppose, to the contrary, that $\sigma^{-1}(0) \neq \emptyset$, i.e., $m > n$.
By \eqref{eq:dw},
$d( g(\vect{e}_1), g(\overline{\vect{e}_2}) )
= d( \vect{e}_1, \overline{\vect{e}_2} )
= n - 2$
and
$d( g(\vect{0}), g(\overline{\vect{e}_{\{1,2\}}}) )
= d( \vect{0}, \overline{\vect{e}_{\{1,2\}}} )
= n - 2$.
If $g$ maps true points of $f_n$ to true points of $f_m$, then $g(\vect{e}_1)$ and $g(\overline{\vect{e}_2} )$ are true points of $f_m$ with Hamming distance $n - 2$.
If $g$ maps false points of $f_n$ to true points of $f_m$, then $g(\vect{0})$ and $g(\overline{\vect{e}_{\{1,2\}}} )$ are true points of $f_m$ with Hamming distance $n - 2$.
By Observation~\ref{obs:distances}, the Hamming distance between true points of $f_m$ must be $0$, $2$, $m - 2$, or $m$.
Because $n \geq 5$, we have $3 \leq n - 2 < m - 2$.
We have reached a contradiction.
\end{pfclaim}

Claim~\ref{clm:fn-Omega1:4} now gives us the desired conclusion that $m = n$.

\ref{lem:fn-Omega1:IoIi}
Since every function in $\clIo$ preserves $0$, $\varphi \in \{ f_m \} \, \clIo$ implies $\varphi(\vect{0}) = f_m(\vect{0}) = 0$, and therefore $\varphi \neq \vak{1}$.
Similarly, every function in $\clIi$ preserves $1$, so $\varphi \in \{ f_m \} \, \clIi$ implies $\varphi(\vect{1}) = f_m(\vect{1}) = 0$, and therefore $\varphi \neq \vak{1}$.
Part \ref{lem:fn-Omega1:Omega1} then gives $m = n$.
\end{proof}

\begin{lemma}
\label{lem:qn-Istar}
Let $m, n \geq 3$.
If $\varphi \in \{ q_m \} \, \clIstar$ and $q_n \minorant \varphi$, then $m = n$.
\end{lemma}

\begin{proof}
Assume $\varphi \in \{ q_m \} \, \clIstar$.
Then $\varphi = q_m (g_1, \dots, g_m)$ for some $g_1, \dots, g_m \in \clIstar$.
Taking $g \colon \{0,1\}^n \to \{0,1\}^m$, $g(\vect{a}) = (g_1(\vect{a}), \dots, g_m(\vect{a})$, we have $\varphi = q_m \circ g$.
We can write $g_i = \pr^{(n)}_{\sigma(i)} + \vak{d_i}$ for some $\sigma \colon \nset{m} \to \nset{n}$ and $d_i \in \{0,1\}$.
Thus, $g(\vect{a}) = \gamma(\vect{a}) + \vect{d}$,
where $\gamma \colon \{0,1\}^n \mapsto \{0,1\}^m$ is $\gamma(\vect{a}) = \vect{a} \sigma + \vect{d}$, where $\vect{a} \sigma = (a_{\sigma(1)}, \dots, a_{\sigma(m)})$ and $\vect{d} = (d_1, \dots, d_m)$.

Now, because $q_n \minorant \varphi$, for every $\vect{a} \in q_n^{-1}(1)$, it holds that $q_m(g(\vect{a})) = 1$, i.e., $g(\vect{a}) \in q_m^{-1}(1)$; in other words, $g$ maps true points of $q_n$ to true points of $q_m$.
We will consider two cases according to the image of $\vect{1}$, a true point of $q_n$, under $g$.
We must have that $g(\vect{q})$ equals either $\vect{1}$ or $\vect{e}_p$ for some $p \in \nset{m}$.

Assume first that $g(\vect{1}) = \vect{e}_p$ for some $p \in \nset{m}$.
Then $\vect{e}_p = \vect{1} \sigma + \vect{d} = \vect{1} + \vect{d}$, so $\vect{d} = \overline{\vect{e}_p}$.
Consequently, $g(\vect{a}) = \vect{a} \sigma + \overline{\vect{e}_p}$ for all $\vect{a}$.
In particular, for each $j \in \nset{n}$, we have $g(\vect{e}_j) = \vect{e}_{\sigma^{-1}(j)} + \overline{\vect{e}_p}$.
Because $g$ maps true points of $q_m$ to true points of $q_n$, we must have that for all $j \in \nset{n}$,
$\sigma^{-1}(j) = \{p\}$ (in which case $g(\vect{e}_j) = \vect{1}$),
$\sigma^{-1}(j) = \nset{m}$ (in which case $g(\vect{e}_j) = \vect{e}_p$,
or
$\sigma^{-1}(j) = \nset{m} \setminus \{p,q\}$ for some $q \in \nset{m}$ with $p \neq q$ (in which case $g(\vect{e}_j) = \vect{e}_q$).
But any collection of three sets of this form are not pairwise disjoint, which implies that $\range \sigma$ has at most two elements because $\sigma$\hyp{}preimages of distinct elements are disjoint.
Because $n \geq 3$, there is an element $r \in \nset{n} \setminus \range \sigma$, and we have
$g(\vect{e}_r) = \vect{e}_r \sigma + \overline{\vect{e}_p} = \vect{0} + \overline{\vect{e}_p} = \overline{\vect{e}_p} \notin q_m^{-1}(1)$.
We have reached a contradiction, which shows that this case is impossible.

Assume now that $g(\vect{1}) = \vect{1}$.
Then $\vect{1} = \vect{1} \sigma + \vect{d} = \vect{1} + \vect{d}$, so $\vect{d} = \vect{0}$.
Consequently, $g(\vect{a}) = \vect{a} \sigma$ for all $\vect{a}$.
In particular, for each $j \in \nset{n}$, we have $g(\vect{e}_j) = \vect{e}_{\sigma^{-1}(j)}$.
Because $g$ maps true points of $q_m$ to true points of $q_n$, we must have that for all $j \in \nset{n}$, $\card{\sigma^{-1}(j)} \in \{1, m\}$; thus $\sigma$ is surjective.
It is not possible that $\card{\sigma^{-1}(j)} = m$ for some $j$, because $n \geq 3$ and $\sigma$\hyp{}preimages of distinct elements are disjoint. 
Consequently, $\card{\sigma^{-1}(j)} = 1$ for all $j \in \nset{n}$, so $\sigma$ is injective.
We conclude that $\sigma$ is bijective, and therefore $m = n$.
\end{proof}

\begin{proposition}
\label{prop:Omega1Omega1}
For $n \geq 5$, $S \subseteq \IN^{+} \setminus \nset{4}$, we have $f_n \in \gen[(\clOmegaOne,\clOmegaOne)]{F_S}$ if and only if $n \in S$.
\end{proposition}

\begin{proof}
If $n \in S$, then obviously $f_n \in \gen[(\clOmegaOne,\clOmegaOne)]{F_S}$.

Assume now that $f_n \in \gen[(\clOmegaOne,\clOmegaOne)]{F_S} = \clOmegaOne ( F_S \, \clOmegaOne )$.
Because $f_n$ is not a constant function, we have that $f_n = \varphi$ or $f_n = \neg \varphi$, where $\varphi = f_m(h_1, \dots, h_m)$ for some $m \in S$ and $h_1, \dots, h_m \in \clOmegaOne$.
It follows from Lemma~\ref{lem:fn-Omega1}\ref{lem:fn-Omega1:Omega1} that $m = n$, and therefore $n \in S$.
\end{proof}

\begin{proposition}
\label{prop:Omega1Lambda}
For $n \geq 5$, $S \subseteq \IN^{+} \setminus \nset{4}$, we have $f_n \in \gen[(\clOmegaOne,\clLambda)]{F_S}$ if and only if $n \in S$.
\end{proposition}

\begin{proof}
If $n \in S$, then obviously $f_n \in \gen[(\clOmegaOne,\clLambda)]{F_S}$.

Assume now that $f_n \in \gen[(\clOmegaOne,\clLambda)]{F_S} = \clLambda ( F_S \, \clOmegaOne )$.
Then $f_n = \bigwedge_{i = 1}^\ell g_i$ for some $g_i \in F_S \, \clOmegaOne$, i.e.,
$g_i = f_{m_i}(h_{i1}, \dots, h_{i m_i})$, where $m_i \in S$ and the inner functions $h_{ij}$ come from $\clOmegaOne$.
Then $f_n \minorant g_i$ for all $i \in \nset{\ell}$.
Because $f_n$ is not the constant function $\vak{1}$, there must exist a $j \in \nset{\ell}$ such that $g_j \neq \vak{1}$.
It follows from Lemma~\ref{lem:fn-Omega1}\ref{lem:fn-Omega1:Omega1} that $m_j = n$, and therefore $n \in S$.
\end{proof}

\begin{proposition}
\label{prop:I0Uinf}
Let $n \geq 5$, $S \subseteq \IN^{+} \setminus \nset{4}$, and $a \in \{0,1\}$.
We have $f_n \in \gen[(\clIa{a},\clUinf)]{F_S}$ if and only if $n \in S$.
\end{proposition}

\begin{proof}
If $n \in S$, then obviously $f_n \in \gen[(\clIa{a},\clUinf)]{F_S}$.

Assume now that $f_n \in \gen[(\clIa{a},\clUinf)]{F_S} = \clUinf ( F_S \, \clIa{a} )$.
Then $f_n = \mu ( g_1, \dots, g_\ell )$ for some $\mu \in \clUinf$ and $g_1, \dots, g_\ell \in F_S \, \clIa{a}$, so for $i \in \nset{\ell}$, $g_i = f_{m_i}( h_{i1}, \dots, h_{i m_i} )$, where $m_i \in S$ and the inner functions $h_{ij}$ come from $\clIa{a}$.
Because $\mu \in \clUinf$, we have $\bigwedge \mu^{-1}(1) \neq \vect{0}$, so there is a $j \in \nset{\ell}$ such that $f_n^{-1}(1) \subseteq g_j^{-1}(1)$, i.e., $f_n \minorant g_j$.
It follows from Lemma~\ref{lem:fn-Omega1}\ref{lem:fn-Omega1:IoIi} that $m_j = n$, and therefore $n \in S$.
\end{proof}

\begin{proposition}
\label{prop:IstarUinf}
Let $n \geq 3$, $S \subseteq \IN^{+} \setminus \nset{2}$.
We have $q_n \in \gen[(\clIstar,\clUinf)]{Q_S}$ if and only if $n \in S$.
\end{proposition}

\begin{proof}
If $n \in S$, then obviously $q_n \in \gen[(\clIstar,\clUinf)]{Q_S}$.

Assume now that $q_n \in \gen[(\clIstar,\clUinf)]{Q_S} = \clUinf ( Q_S \, \clIstar )$.
Then $q_n = \mu ( g_1, \dots, g_\ell )$ for some $\mu \in \clUinf$ and $g_1, \dots, g_\ell \in Q_S \, \clIstar$, so for $i \in \nset{\ell}$, $g_i = q_{m_i}( h_{i1}, \dots, h_{i m_i} )$, where $m_i \in S$ and the inner functions $h_{ij}$ come from $\clIstar$.
Because $\mu \in \clUinf$, we have $\bigwedge \mu^{-1}(1) \neq \vect{0}$, so there is a $j \in \nset{\ell}$ such that $q_n^{-1}(1) \subseteq g_j^{-1}(1)$, i.e., $q_n \minorant g_j$.
It follows from Lemma~\ref{lem:qn-Istar} that $m_j = n$, and therefore $n \in S$.
\end{proof}

\begin{lemma}
\label{lem:fmL}
Let $m$ and $n$ be even integers greater than or equal to $6$.
If $f_n \in \{ f_m \} \clL$ or $\neg f_n \in \{ f_m \} \clL$, then $m = n$.
\end{lemma}

\begin{proof}
Assume that $\varphi \in \{ f_m \} \clL$.
Then $\varphi = f_m ( g_1, \dots, g_m )$ for some $g_1, \dots, g_m \in \clL$.

\begin{claim}
\label{lem:fmL:clm:g-prop}
The map $\mathbf{g} = (g_1, \dots, g_n) \colon \{0,1\}^n \to \{0,1\}^m$ has the following properties:
\begin{enumerate}[label=\upshape{(\roman*)}]
\item\label{lem:fmL:clm:g-prop:hom}
For the ternary group, $\mathbf{B} = (\{0,1\}, \mathord{+_3})$ with $\mathord{+_3}(x,y,z) = x + y + z$, $\mathbf{g}$ is a homomorphism from $\mathbf{B}^n$ to $\mathbf{B}^m$, and hence for every odd natural number $2k + 1$,
\begin{equation}
\displaybump
\mathbf{g} \left( \sum_{i = 1}^{2k + 1} \vect{u}_i \right) = \sum_{i = 1}^{2k + 1} \mathbf{g}(\vect{u}_i).
\label{eq:g-odd}
\end{equation}

\item\label{lem:fmL:clm:g-prop:ei}
If $i \in \nset{n}$, then
\[
\displaybump
\mathbf{g}(\overline{\vect{e}_i}) = \sum_{j \in \nset{n} \setminus \{i\}} \mathbf{g}(\vect{e}_j)
\qquad\text{and}\qquad
\mathbf{g}(\vect{e}_i) = \sum_{j \in \nset{n} \setminus \{i\}} \mathbf{g}(\overline{\vect{e}_j}).
\]

\item\label{lem:fmL:clm:g-prop:im}
$\mathbf{g}$ maps $\varphi^{-1}(1)$ into $f_m^{-1}(1)$ and $\varphi^{-1}(0)$ into $f_m^{-1}(0)$.
\end{enumerate}
\end{claim}

\begin{pfclaim}[Proof of Claim~\ref{lem:fmL:clm:g-prop}]
\ref{lem:fmL:clm:g-prop:hom}
The clone $\clL$ is generated by $+$ and $1$.
Since $+$ is a homomorphism $\mathbf{B}^2 \to \mathbf{B}$ and $1$ is a homomorphism $\mathbf{B} \to \mathbf{B}$, it follows that $\mathbf{g}$ is a homomorphism $\mathbf{B}^n \to \mathbf{B}^m$.
This means that $\mathbf{g}(\vect{u} + \vect{v} + \vect{w}) = \mathbf{g}(\vect{u}) + \mathbf{g}(\vect{v}) + \mathbf{g}(\vect{w})$.
Repeated application of this equality yields \eqref{eq:g-odd}.

\ref{lem:fmL:clm:g-prop:ei}
Because $n$ is even, we have $\overline{\vect{e}_i} = \sum_{j \in \nset{n} \setminus \{i\}} \vect{e}_j$ and $\vect{e}_i = \sum_{j \in \nset{n} \setminus \{i\}} \overline{\vect{e}_j}$, and we get the claimed equalities by \ref{lem:fmL:clm:g-prop:hom}.

\ref{lem:fmL:clm:g-prop:im}
For any $\vect{a} \in \varphi^{-1}(1)$ we have $(f_m \circ \mathbf{g})(\vect{a}) = 1$; hence $\mathbf{g}(\vect{a}) \in f_m^{-1}(1)$.
Similarly, for any $\vect{a} \in \varphi^{-1}(0)$ we have $\mathbf{g}(\vect{a}) \in f_m^{-1}(0)$.
\end{pfclaim}

Assume from now on that $\varphi \in \{f_n, \overline{f_n}\}$.

\begin{claim}
\label{lem:fmL:clm:e}
\leavevmode
\begin{enumerate}[label=\upshape{(\roman*)}]
\item\label{lem:fmL:clm:e:e}
For all $i, j \in \nset{n}$ with $i \neq j$, we have $\mathbf{g}(\vect{e}_i) \neq \mathbf{g}(\vect{e}_j)$.
\item\label{lem:fmL:clm:e:e-compl}
For all $i, j \in \nset{n}$ with $i \neq j$, we have $\mathbf{g}(\overline{\vect{e}_i}) \neq \mathbf{g}(\overline{\vect{e}_j})$.
\end{enumerate}
\end{claim}

\begin{pfclaim}[{Proof of Claim~\ref{lem:fn-V:e}}]
\ref{lem:fmL:clm:e:e}
Suppose, to the contrary, that there are $i, j \in \nset{n}$ with $i \neq j$ such that $\mathbf{g}(\vect{e}_i) = \mathbf{g}(\vect{e}_j)$.
Let $k \in \nset{n} \setminus \{i,j\}$.
Then $\vect{e}_{ijk} = \vect{e}_i + \vect{e}_j + \vect{e}_k$, and by \eqref{eq:g-odd},
\begin{align*}
\mathbf{g}(\vect{e}_{ijk})
= \mathbf{g}(\vect{e}_i) + \mathbf{g}(\vect{e}_j) + \mathbf{g}(\vect{e}_k)
= \mathbf{g}(\vect{e}_k),
\end{align*}
which contradicts Claim~\ref{lem:fmL:clm:g-prop}\ref{lem:fmL:clm:g-prop:im} because $f_n(\vect{e}_{ijk}) = 0$ and $f_n(\vect{e}_k) = 1$ and hence $\varphi(\vect{e}_{ijk}) \neq \varphi(\vect{e}_k)$.

\ref{lem:fmL:clm:e:e-compl}
The proof is similar to \ref{lem:fmL:clm:e:e}. We just need to replace each $\vect{e}_I$ with $\overline{\vect{e}_I}$.
\end{pfclaim}

\begin{claim}
\label{lem:fmL:clm:eneg}
There are no $i, j \in \nset{n}$ such that $\mathbf{g}(\vect{e}_i) = \overline{\mathbf{g}(\vect{e}_j)}$ or $\mathbf{g}(\overline{\vect{e}_i}) = \overline{\mathbf{g}(\overline{\vect{e}_j})}$.
\end{claim}

\begin{pfclaim}[Proof of Claim~\ref{lem:fmL:clm:eneg}]
Suppose, to the contrary, that $\mathbf{g}(\vect{e}_i) = \overline{\mathbf{g}(\vect{e}_j)}$.
Then $i \neq j$ and $\mathbf{g}(\vect{e}_i) + \overline{\mathbf{g}(\vect{e}_j)} = \vect{1}$.
Let $k \in \nset{n} \setminus \{i, j\}$.
We get
\[
\mathbf{g}(\vect{e}_{ijk})
= \mathbf{g}(\vect{e}_i) + \mathbf{g}(\vect{e}_j) + \mathbf{g}(\vect{e}_k)
= \vect{1} + \mathbf{g}(\vect{e}_k)
= \overline{\mathbf{g}(\vect{e}_k)}.
\]
Because $f_m$ is reflexive, i.e., $f_m(\vect{a}) = f_m(\overline{\vect{a}})$ for all $\vect{a} \in \{0,1\}^m$, this implies that $f_m(\mathbf{g}(\vect{e}_{ijk})) = f_m(\mathbf{g}(\vect{e}_k))$.
This contradicts Claim~\ref{lem:fmL:clm:g-prop}\ref{lem:fmL:clm:g-prop:im} because $f_n(\vect{e}_{ijk}) = 0$ and $f_n(\vect{e}_k) = 1$ and hence $\varphi(\vect{e}_{ijk}) \neq \varphi(\vect{e}_k)$.

The claim about $\overline{\vect{e}_i}$ and $\overline{\vect{e}_j}$ is proved in a similar way.
\end{pfclaim}

\begin{claim}
\label{lem:fmL:clm:g-inj1}
There are no $\vect{u}, \vect{v} \in \{0,1\}^n$ such that $\vect{u} \neq \vect{v}$, $\vect{u} \neq \overline{\vect{v}}$, $w(\vect{u}) \notin \{1, n-1\}$, $w(\vect{v}) \notin \{1, n-1\}$, and $\mathbf{g}(\vect{u}) = \mathbf{g}(\vect{v})$ or $\mathbf{g}(\vect{u}) = \overline{\mathbf{g}(\vect{v})}$.
\end{claim}

\begin{pfclaim}[Proof of Claim~\ref{lem:fmL:clm:g-inj1}]
Suppose, to the contrary, that $\vect{u}, \vect{v} \in \{0,1\}^n$ are tuples such that $\vect{u} \neq \vect{v}$, $\vect{u} \neq \overline{\vect{v}}$, $w(\vect{u}) \notin \{1, n-1\}$, $w(\vect{v}) \notin \{1, n-1\}$, and $\mathbf{g}(\vect{u}) = \mathbf{g}(\vect{v}) + \vect{c}$ for some $\vect{c} \in \{\vect{0}, \vect{1}\}$.
Let $I := \{ \, i \in \nset{n} \mid u_i = 1 \,\}$ and $J := \{ \, i \in \nset{n} \mid v_i = 1 \, \}$; then $\vect{u} = \vect{e}_I$, $\vect{v} = \vect{e}_J$, $I \neq J$ and $I \neq \nset{n} \setminus J$.
The following sets are pairwise disjoint and their union is $\nset{n}$:
\begin{align*}
A &:= I \cap J, &
B &:= I \setminus J, &
C &:= J \setminus I, &
D &:= \nset{n} \setminus (I \cup J).
\end{align*}
Let $K: = B \cup C = I \triangle J$ (the symmetric difference of $I$ and $J$).
Because $I \neq J$, we have $K \neq \emptyset$.
Because $I \neq \nset{n} \setminus J$, we have $A \cup D = \nset{n} \setminus K \neq \emptyset$.
Therefore, $1 \leq \card{K} \leq n - 1$.

Choose an element $k \in \nset{n}$ as follows.
If $1 \leq \card{K} \leq 2$, then take $k$ from $A \cup D = \nset{n} \setminus K$.
If $n - 2 \leq \card{K} \leq n - 1$, then take $k$ from $K$.
Otherwise $k$ is arbitrary.
Let $K' := K \triangle \{k\}$.
The choice of $k$ ensures that $2 \leq \card{K'} \leq n - 2$; hence, $\vect{e}_K + \vect{e}_k = \vect{e}_{K \triangle \{k\}} = \vect{e}_{K'} \in f_n^{-1}(0)$.
Moreover, $\vect{e}_k \in f_n^{-1}(1)$, so $\varphi(\vect{e}_{K'}) \neq \varphi(\vect{e}_k)$.
However,
\[
\mathbf{g}(\vect{e}_k)
= \mathbf{g}(\vect{e}_I + \vect{e}_J + \vect{e}_{K'})
= \mathbf{g}(\vect{e}_I) + \mathbf{g}(\vect{e}_J) + \mathbf{g}(\vect{e}_{K'})
= \mathbf{g}(\vect{e}_{K'}) + \vect{c}.
\]
Because $f_m$ is reflexive, this implies $f_m(\mathbf{g}(\vect{e}_k)) = f_m(\mathbf{g}(\vect{e}_{K'}) + \vect{c}) = f_m(\mathbf{g}(\vect{e}_{K'}))$,
which contradicts Claim~\ref{lem:fmL:clm:g-prop}\ref{lem:fmL:clm:g-prop:im}.
\end{pfclaim}

The proof now proceeds in different ways according to whether $\varphi = f_n$ or $\varphi = \overline{f_n}$.
We assume first that $\varphi = f_n$.
By Claim~\ref{lem:fmL:clm:g-prop}\ref{lem:fmL:clm:g-prop:im}, there exists a map $\sigma \colon \nset{n} \to \nset{m}$ such that $\sigma(i) = j$ if and only if $\mathbf{g}(\vect{e}_i) \in \{\vect{e}_j, \overline{\vect{e}_j}\}$.
Thus, for every $i \in \nset{n}$, there exists a $\vect{c}_i \in \{\vect{0}, \vect{1}\}$ such that $\mathbf{g}(\vect{e}_i) = \vect{e}_{\sigma(i)} + \vect{c}_i$.
By Claims \ref{lem:fmL:clm:e} and \ref{lem:fmL:clm:eneg}, $\sigma$ is injective, so $m \geq n$.
Let now
\begin{equation}
\vect{d} := \sum_{i = 1}^n \mathbf{g}(\vect{e}_i).
\label{eq:d}
\end{equation}

\begin{claim}
\label{lem:fmL:clm:d}
For the tuple $\vect{d}$ defined above in \eqref{eq:d}, $\vect{d} = \vect{0}$ or $\vect{d} = \vect{1}$.
\end{claim}

\begin{pfclaim}[Proof of Claim~\ref{lem:fmL:clm:d}]
Note that
\[
\vect{d} = \mathbf{g}(\vect{e}_i) + \mathbf{g}(\overline{\vect{e}_i})
\qquad
\text{for all $i \in \nset{n}$.}
\]
Now, fix a $p \in \nset{n}$, and let $\vect{u} := \mathbf{g}(\vect{e}_p)$ and $\vect{v} := \mathbf{g}(\overline{\vect{e}_p})$.
Then $\vect{d} = \vect{u} + \vect{v}$.

By Claim~\ref{lem:fmL:clm:g-prop}\ref{lem:fmL:clm:g-prop:im}, there exist $r, s \in \nset{m}$ such that $\vect{u} \in \{\vect{e}_r, \overline{\vect{e}_r}\}$ and $\vect{v} \in \{\vect{e}_s, \overline{\vect{e}_s}\}$.
Suppose, to the contrary, that $\vect{d} \neq \vect{0}$ and $\vect{d} \neq \vect{1}$, that is, $\vect{u} \neq \vect{v}$ and $\vect{u} \neq \overline{\vect{v}}$, i.e., $r \neq s$.
Then $\vect{d} \in \{\vect{e}_{rs}, \overline{e}_{rs}\}$.
In $\vect{d} = \vect{e}_{rs}$, then the only possible decompositions of $\vect{d}$ into elements of $f_m^{-1}(1)$ are $\vect{e}_r + \vect{e}_s$ and $\overline{\vect{e}_r} + \overline{\vect{e}_s}$.
In $\vect{d} = \overline{\vect{e}_{rs}}$, then the only possible decompositions of $\vect{d}$ into elements of $f_m^{-1}(1)$ are $\vect{e}_r + \overline{\vect{e}_s}$ and $\overline{\vect{e}_r} + \vect{e}_s$.
Consequently, either for all $i \in \nset{n}$, $\{\mathbf{g}(\vect{e}_i), \mathbf{g}(\overline{\vect{e}_i})\} \in \{ \{ \vect{e}_r, \vect{e}_s\}, \{ \overline{\vect{e}_r}, \overline{\vect{e}_s} \} \}$,
or for all $i \in \nset{n}$, $\{\mathbf{g}(\vect{e}_i), \mathbf{g}(\overline{\vect{e}_i})\} \in \{ \{ \vect{e}_r, \overline{\vect{e}_s}\}, \{ \overline{\vect{e}_r}, \vect{e}_s \} \}$.
Because $n \geq 6$, by the pigeonhole principle, there exist $i, j \in \nset{n}$ with $i \neq j$ such that $\mathbf{g}(\vect{e}_i) = \mathbf{}(\vect{e}_j)$, which contradicts the injectivity of $\sigma$.
\end{pfclaim}

By Claim~\ref{lem:fmL:clm:d}, for all $i \in \nset{n}$, $\{ \mathbf{g}(\vect{e}_i), \mathbf{g}(\overline{\vect{e}_i}) \} \subseteq \{ \vect{e}_{\sigma(i)}, \overline{\vect{e}_{\sigma(i)}} \}$.

Suppose now, to the contrary, that $m > n$.
Let $i \in \nset{n}$.
By Claim~\ref{lem:fmL:clm:g-prop}\ref{lem:fmL:clm:g-prop:ei},
\[
\mathbf{g}(\overline{\vect{e}_i})
= \sum_{j \in \nset{n} \setminus \{j\}} \mathbf{g}(\vect{e}_j)
= \sum_{j \in \nset{n} \setminus \{j\}} (\vect{e}_{\sigma(j)} + \vect{c}_j)
= \bigl( \underbrace{\sum_{j \in \nset{n} \setminus \{j\}} \vect{e}_{\sigma(j)}}_{\vect{t}} \bigr) + \vect{c},
\]
for some $\vect{c} \in \{\vect{0}, \vect{1}\}$.
We have $w(\vect{t}) = n - 1$, so $w(\vect{t} + \vect{c}) \in \{n - 1, m - n + 1 \}$.
Because $m > n$, we have $2 \leq w(\vect{t} + \vect{c}) \leq n - 2$,
so $\mathbf{g}(\overline{\vect{e}_i}) \in f_m^{-1}(0)$, a contradiction.
We conclude that $n = m$ whenever $f_n \in \{f_m\} \clL$.

Assume now that $\varphi = \overline{f_n}$.
By Claim~\ref{lem:fmL:clm:g-prop}\ref{lem:fmL:clm:g-prop:im}, $\mathbf{g}$ maps $f_n^{-1}(1)$ into $f_m^{-1}(0)$ and $f_n^{-1}(0)$ into $f_m^{-1}(1)$.
For $i \in \{1, 2, 3\}$, we have $f_n(\vect{e}_i + \vect{e}_4) = 0$, so $\mathbf{g}(\vect{e}_i + \vect{e}_4) = \vect{e}_{j_i} + \vect{c}_i$ for some $j_i \in \nset{m}$ and $\vect{c}_i \in \{ \vect{0}, \vect{1} \}$.
By Claim~\ref{lem:fmL:clm:g-inj1}, $j_1$, $j_2$, and $j_3$ are pairwise distinct.
Now,
\begin{align*}
&
\sum_{i=1}^3 (\vect{e}_i + \vect{e}_4)
= \vect{e}_{\{1,2,3,4,\}} \in f_n^{-1}(0),
\\ &
\mathbf{g} \bigl( \sum_{i=1}^3 (\vect{e}_i + \vect{e}_4) \bigr)
= \sum_{i=1}^3 \mathbf{g}(\vect{e}_i + \vect{e}_4)
= \vect{e}_{\{j_1,j_2,j_3\}} \in f_m^{-1}(0),
\end{align*}
which contradicts Claim~\ref{lem:fmL:clm:g-prop}\ref{lem:fmL:clm:g-prop:im}.
\end{proof}

\begin{proposition}
\label{prop:LOmega1}
Let $M := \{ \, n \in \IN \mid n \geq 6, \, \text{$n$ even} \, \}$. 
Let $n \in M$ and $S \subseteq M$.
We have $f_n \in \gen[(\clL,\clOmegaOne)]{F_S}$ if and only if $n \in S$.
\end{proposition}

\begin{proof}
If $n \in S$, then obviously $f_n \in \gen[(\clL,\clOmegaOne)]{F_S}$.

Assume now that $f_n \in \gen[(\clL,\clOmegaOne)]{F_S} = \clOmegaOne ( F_S \, \clL )$.
Then $f_n = u ( g_1, \dots, g_\ell )$ for some $u \in \clOmegaOne$ and $g_1, \dots, g_\ell \in F_S \, \clL$.
Clearly, $u$ cannot be a constant function, so $u$ is either a projection or a negated projection.
Therefore, for some $i \in \nset{\ell}$,
$f_n = g_i$ or $f_n = \overline{g_i}$, where $g_i = f_m(h_1, \dots, h_m)$ for some $m \in S$ and $h_1, \dots, h_m \in \clL$.
It follows from Lemma~\ref{lem:fmL} that $n = m \in S$.
\end{proof}

\begin{lemma}
\label{lem:fn-V}
Let $m, n \geq 4$.
If $\varphi \in \{ f_m \} \, \clV$, $\varphi(\vect{1}) = 0$, and $f_n \minorant \varphi$, then $m = n$.
\end{lemma}

\begin{proof}
Assume $\varphi \in \{ f_m \} \, \clV$ and $f_n \minorant \varphi$.
Then $\varphi = f_m \circ \mathbf{g}$, where $\mathbf{g} \in (\clV^{(n)})^m$.

Let $\mathbf{B} = (\{0,1\}, \mathord{\vee})$ be the semigroup with neutral element $0$.

\begin{claim}
\label{lem:fn-V:g-prop}
The function $\mathbf{g}$ has the following properties.
\begin{enumerate}[label=\upshape{(\roman*)}]
\item\label{lem:fn-V:g-prop:hom}
$\mathbf{g}$ is a homomorphism from $\mathbf{B}^n$ to $\mathbf{B}^m$, i.e., $\mathbf{g}(\vect{u} \vee \vect{v}) = \mathbf{g}(\vect{u}) \vee \mathbf{g}(\vect{v})$ for all $\vect{u}, \vect{v} \in \mathbf{B}^n$.

\item
\label{lem:fn-V:g-prop:sum}
If $i \in \nset{n}$, then
\[
\displaybump
\mathbf{g}(\overline{\vect{e}_i}) = \bigvee_{j \in \nset{n} \setminus \{i\}} \mathbf{g}(\vect{e}_j).
\]

\item
\label{lem:fn-V:g-prop:im}
$\mathbf{g}$ maps $f_n^{-1}(1)$ into $f_m^{-1}(1)$.
\end{enumerate}
\end{claim}

\begin{pfclaim}[{Proof of Claim~\ref{lem:fn-V:g-prop}}]
\ref{lem:fn-V:g-prop:hom}
The clone $\clV$ is generated by $\vee$ and the constants $0$ and $1$.
The claim follows from the fact that $\vee$ is a homomorphism $\mathbf{B}^2 \to \mathbf{B}$ and $0$ and $1$ are homomorphisms $\mathbf{B} \to \mathbf{B}$.

\ref{lem:fn-V:g-prop:sum}
This follows from the homomorphism property \ref{lem:fn-V:g-prop:hom} and from the fact that $\overline{\vect{e}_i} = \bigvee_{j \in \nset{n} \setminus \{i\}} \mathbf{e}_j$.

\ref{lem:fn-V:g-prop:im}
Because $f_n \minorant f_m \circ \mathbf{g}$, we have that $f_n^{-1}(1) \subseteq (f_m \circ \mathbf{g})^{-1}(1)$.
Consequently, $\mathbf{g}$ maps $f_n^{-1}(1)$ into $f_m^{-1}(1)$.
\end{pfclaim}

\begin{claim}
\label{lem:fn-V:e}
\leavevmode
\begin{enumerate}[label=\upshape{(\roman*)}]
\item\label{lem:fn-V:e:e-compl}
For all $i, j \in \nset{n}$ with $i \neq j$, we have $\mathbf{g}(\overline{\vect{e}_i}) \neq \mathbf{g}(\overline{\vect{e}_j})$.
\item\label{lem:fn-V:e:e}
For all $i, j \in \nset{n}$ with $i \neq j$, we have $\mathbf{g}(\vect{e}_i) \neq \mathbf{g}(\vect{e}_j)$.
\end{enumerate}
\end{claim}

\begin{pfclaim}[{Proof of Claim~\ref{lem:fn-V:e}}]
\ref{lem:fn-V:e:e-compl}
Suppose, to the contrary, that there are $i, j \in \nset{n}$ with $i \neq j$ such that $\mathbf{g}(\overline{\vect{e}_i}) = \mathbf{g}(\overline{\vect{e}_j}) =: \vect{d}$.
We have $\vect{d} \in f_m^{-1}(1)$ by Claim~\ref{lem:fn-V:g-prop}\ref{lem:fn-V:g-prop:im}.
Moreover, $\overline{\vect{e}_i} \vee \overline{\vect{e}_j} = \vect{1}$, so
\[
\mathbf{g}(\vect{1}) =
\mathbf{g}(\overline{\vect{e}_i} \vee \overline{\vect{e}_j}) =
\mathbf{g}(\overline{\vect{e}_i}) \vee \mathbf{g}(\overline{\vect{e}_j}) =
\vect{d} \vee \vect{d} = \vect{d}.
\]
Therefore, $\varphi(\vect{1}) = f_m(\mathbf{g}(\vect{1})) = f_m(\vect{d}) = 1$,
which contradicts our assumption that $\varphi(\vect{1}) = 0$.

\ref{lem:fn-V:e:e}
Suppose, to the contrary, that there are $i, j \in \nset{n}$ with $i \neq j$ such that $\mathbf{g}(\vect{e}_i) = \mathbf{g}(\vect{e}_j)$.
Then, by Claim~\ref{lem:fn-V:g-prop}\ref{lem:fn-V:g-prop:sum},
\begin{align*}
\mathbf{g}(\overline{\vect{e}_i})
= \bigvee_{k \in \nset{n} \setminus \{i\}} \mathbf{g}(\vect{e}_k)
& = \mathbf{g}(\vect{e}_j) \vee \bigvee_{k \in \nset{n} \setminus \{i,j\}} \mathbf{g}(\vect{e}_k)
\\ &= \mathbf{g}(\vect{e}_i) \vee \bigvee_{k \in \nset{n} \setminus \{i,j\}} \mathbf{g}(\vect{e}_k)
= \bigvee_{k \in \nset{n} \setminus \{j\}} \mathbf{g}(\vect{e}_k)
= \mathbf{g}(\overline{\vect{e}_j}),
\end{align*}
which is a contradiction to part \ref{lem:fn-V:e:e-compl}.
\end{pfclaim}

\begin{claim}
\label{lem:fn-V:e-e-compl}
For all $i \in \nset{n}$ and for all $j \in \nset{m}$, $\mathbf{g}(\vect{e}_i) \neq \overline{\vect{e}_j}$.
\end{claim}

\begin{pfclaim}[{Proof of Claim~\ref{lem:fn-V:e-e-compl}}]
Suppose, to the contrary, that there are $i \in \nset{n}$ and $j \in \nset{m}$ such that $\mathbf{g}(\vect{e}_i) = \overline{\vect{e}_j}$.
Then, by Claim~\ref{lem:fn-V:g-prop}\ref{lem:fn-V:g-prop:sum}, we get for all $\ell \in \nset{n} \setminus \{i\}$ that
\begin{align*}
\mathbf{g}(\overline{\vect{e}_\ell})
= \bigvee_{k \in \nset{n} \setminus \{\ell\}} \mathbf{g}(\vect{e}_k)
= \mathbf{g}(\vect{e}_i) \vee \bigvee_{k \in \nset{n} \setminus \{\ell, i\}} \mathbf{g}(\vect{e}_k)
= \overline{\vect{e}_j} \vee \bigvee_{k \in \nset{n} \setminus \{\ell, i\}} \mathbf{g}(\vect{e}_k)
\geq \overline{\vect{e}_j},
\end{align*}
where $\leq$ denotes the natural ordering of the semilattice $\mathbf{B}^m$ induced by the ordering $0 < 1$ of $\mathbf{B}$.
Since the only tuples that are greater than or equal to $\overline{\vect{e}_j}$ are $\overline{\vect{e}_j} \in f_m^{-1}(1)$ and $\vect{1} \in f_m^{-1}(0)$, it follows from Claim~\ref{lem:fn-V:g-prop}\ref{lem:fn-V:g-prop:im} that $\mathbf{g}(\overline{\vect{e}_\ell}) = \overline{\vect{e}_j}$ for all $\ell \in \nset{n} \setminus \{i\}$.
This contradicts Claim~\ref{lem:fn-V:e}, because $n \geq 4$.
\end{pfclaim}

It follows from Claims \ref{lem:fn-V:g-prop}\ref{lem:fn-V:g-prop:im} and \ref{lem:fn-V:e-e-compl} that there is a map $\sigma \colon \nset{n} \to \nset{m}$ such that $\mathbf{g}(\vect{e}_i) = \mathbf{e}_{\sigma(i)}$ for all $i \in \nset{n}$.
By Claim~\ref{lem:fn-V:e}\ref{lem:fn-V:e:e}, $\sigma$ is injective; hence $m \geq n$.
Let now $i \in \nset{n}$, and let $\vect{b} := \mathbf{g}(\overline{\vect{e}_i})$.
Since $\overline{\vect{e}_i} \in f_n^{-1}(1)$, Claim~\ref{lem:fn-V:g-prop}\ref{lem:fn-V:g-prop:im} implies that $\vect{b} \in f_m^{-1}(1)$.
By Claim~\ref{lem:fn-V:g-prop}\ref{lem:fn-V:g-prop:sum} we have
\[
\vect{b}
= \mathbf{g}(\overline{\vect{e}_i})
= \bigvee_{j \in \nset{n} \setminus \{i\}} \mathbf{g}(\vect{e}_j)
= \bigvee_{j \in \nset{n} \setminus \{i\}} \vect{e}_{\sigma(j)}.
\]
Because $\sigma$ is injective, exactly $n - 1$ entries of $\vect{b}$ are equal to $1$.
In order for $\vect{b}$ to be in $f_m^{-1}(1)$, we need to have $w(\vect{b}) \in \{1, m-1\}$.
Thus $n \in \{2, m\}$.
Because we have assumed that $n \geq 4$,
we conclude that $n = m$.
\end{proof}

\begin{proposition}
\label{prop:VLambda}
For $n \geq 4$, $S \subseteq \IN^{+} \setminus \nset{3}$, we have $f_n \in \gen[(\clV,\clLambda)]{F_S}$ if and only if $n \in S$.
\end{proposition}

\begin{proof}
If $n \in S$, then obviously $f_n \in \gen[(\clV,\clLambda)]{F_S}$.

Assume now that $f_n \in \gen[(\clV,\clLambda)]{F_S} = \clLambda ( F_S \clV )$.
Then $f_n = \lambda (\varphi_1, \dots, \varphi_p)$ for some $\lambda \in \clLambda$ and $\varphi_1, \dots, \varphi_p \in F_S \clV$.
Because $f_n$ is not a constant function, $\lambda$ is not constant.
Without loss of generality, we may assume that $\lambda = x_1 \wedge x_2 \wedge \dots \wedge x_p$, and so $f_n = \bigwedge_{i \in \nset{p}} \varphi_i$.
Consequently, $f_n$ is a minorant of each $\varphi_i$ ($i \in \nset{p}$).
Moreover, for each $\vect{a} \in f_n^{-1}(0)$, there is a $j \in \nset{p}$ such that $\varphi_j(\vect{a}) = 0$.
In particular, there is a $q \in \nset{p}$ such that $\varphi_q(\vect{1}) = 0$.
Because $\varphi_q \in \{ f_m \} \clV$ for some $m \in S$, it follows from Lemma~\ref{lem:fn-V} that $n = m \in S$.
\end{proof}

\begin{definition}
Let $f \in \clAll^{(n)}$.
A chain $\vect{u}_0 < \vect{u}_1 < \dots < \vect{u}_\ell$ of tuples in $\{0,1\}^n$ is \emph{alternating} in $f$ if for all $i \in \{0, \dots, \ell-1\}$, $f(\vect{u}_i) \neq f(\vect{u}_{i+1})$.
The number $\ell$ is the \emph{length} of the chain.
The \emph{alternation number} of $f$, denoted $\Alt(f)$, is the length of the longest alternating chain in $f$.
\end{definition}

\begin{proposition}
\label{prop:fi-VJ}
Let $n \geq 4$ and $S \subseteq \IN^{+} \setminus \nset{3}$.
\begin{enumerate}[label=\upshape{(\roman*)}]
\item\label{prop:fi-VJ:VJ}
$\overline{f_n} \notin \gen[(\clV,\clIc)]{F_S}$.
\item\label{prop:fi-VJ:VIstar-neg}
$\overline{f_n} \in \gen[(\clV,\clIstar)]{F_S}$ if and only if $n \in S$.
\item\label{prop:fi-VJ:VOmega1}
$f_n \in \gen[(\clV,\clOmegaOne)]{F_S}$ if and only if $n \in S$.
\end{enumerate}
\end{proposition}

\begin{proof}
\ref{prop:fi-VJ:VJ}
Observe that $\Alt(f_m) = \Alt(\overline{f_m}) = 4$ for all $m \geq 4$.
By \cite[Proposition~6.4]{Lehtonen-discmono}, $f \minor_\clM g$ if and only if $\Alt(f) < \Alt(g)$ or $(\Alt(f), f(\vect{0})) = (\Alt(g), g(\vect{0}))$.
Consequently, for all $n \in \IN$ and for all $S \subseteq \IN$, $\overline{f_n} \notin F_S \clM \supseteq F_S \clV = \gen[(\clV,\clIc)]{F_S}$.

\ref{prop:fi-VJ:VIstar-neg}
Because
\[
\gen[(\clV,\clIstar)]{F_S}
= \clIstar ( F_S \clV )
= ( \clIc \cup \overline{\clIc} ) ( F_S \clV )
= \clIc ( F_S \clV ) \cup \overline{\clIc} ( F_S \clV )
= \gen[(\clV,\clIc)]{F_S} \cup \overline{\gen[(\clV,\clIc)]{F_S}},
\]
it follows from part \ref{prop:fi-VJ:VJ} that $\overline{f_n} \in \gen[(\clV,\clIstar)]{F_S}$ if and only if $f_n \in \gen[(\clV,\clIc)]{F_S}$.
It follows from Proposition~\ref{prop:VLambda} that this holds if and only if $n \in S$.

\ref{prop:fi-VJ:VOmega1}
It is clear that if $n \in S$, then $f_n \in \gen[(\clV,\clOmegaOne)]{F_S}$.
Conversely, if
$f_n \in \gen[(\clV,\clOmegaOne)]{F_S}$,
then $f_n = u (g_1, \dots, g_m)$ for some $u \in \clOmegaOne$ and $g_1, \dots, g_m \in F_S \clV$.
Clearly, $u$ cannot be a constant function, so $u \in \clIstar$, and we have
\[
f_n \in
\clIstar ( F_S \clV )
= ( \clIc \cup \overline{\clIc} ) ( F_S \clV )
= \clIc ( F_S \clV ) \cup \overline{\clIc} ( F_S \clV )
= \gen[(\clV,\clIc)]{F_S} \cup \overline{\gen[(\clV,\clIc)]{F_S}}.
\]
Then either $f_n$ of $\overline{f_n}$ is in $\gen[(\clV,\clIc)]{F_S}$
It follows from Proposition~\ref{prop:VLambda} and part \ref{prop:fi-VJ:VIstar-neg} that $n \in S$.
\end{proof}

\begin{proof}[Proof of Theorem~\ref{thm:uncountable}]
It suffices to prove the statement for $(C_1,C_2) = (K_1,K_2)$, where
$(K_1,K_2)$ is one of the pairs
$(\clOmegaOne,\clOmegaOne)$,
$(\clOmegaOne,\clLambda)$,
$(\clIstar,\clUinf)$,
$(\clIo,\clUinf)$,
$(\clIi,\clUinf)$,
$(\clL,\clOmegaOne)$,
$(\clV,\clLambda)$,
$(\clV,\clOmegaOne)$.
By Lemma~\ref{lem:clonmon}, it holds that for all clones $C_1$ and $C_2$ such that $C_1 \subseteq K_1$ and $C_2 \subseteq K_2$, there are an uncountable infinitude of $(C_1,C_2)$\hyp{}clonoids.
For the remaining pairs $(K_1,K_2)$ listed in the statement of the theorem, the result follows from Proposition~\ref{prop:Knid}.

By Propositions
\ref{prop:Omega1Omega1},
\ref{prop:Omega1Lambda},
\ref{prop:I0Uinf},
\ref{prop:IstarUinf},
\ref{prop:LOmega1},
\ref{prop:VLambda},
\ref{prop:fi-VJ},
there exists a countably infinite set $F$ of functions
with the property that for all subsets $S \subseteq F$ and for all $f \in F$,
we have $f \in \gen[(K_1,K_2)]{S}$ if and only if $f \in S$.
Consequently, for all $S, T \subseteq F$,
we have $\gen[(K_1,K_2)]{S} = \gen[(K_1,K_2)]{T}$ if and only if $S = T$.
Because the power set of $F$ is uncountable, it follows that there are an uncountable infinitude of $(K_1,K_2)$\hyp{}clonoids.
\end{proof}


\section{Clonoids with source and target clones of $0$- or $1$-separating functions}

In this section, we will show that there are an uncountable infinitude of $(C_1,C_2)$\hyp{}clonoids when $C_1$ and $C_2$ are subclones of clones of $0$- or $1$\hyp{}separating functions of rank $2$ and $\infty$, respectively.
We refer here to the definitions and results presented in an earlier paper by Ne\v{s}et\v{r}il and the present author \cite{LehNes-clique}.
In this earlier paper, the main notion of study was the so\hyp{}called $C$\hyp{}minor relation on $\mathcal{F}_{AB}$, where $C$ is a clone on $A$.
We will first briefly explain the connection between $C$\hyp{}minors and $(C_1,C_2)$\hyp{}clonoids in order to translate the earlier results into our current setting.

\begin{definition}
Let $f, g \in \mathcal{F}_{AB}$, and let $C$ be a clone on $A$.
We say that $f$ is a \emph{$C$\hyp{}minor} of $g$, and we write $f \leq_C g$, if $f \in \{g\} C$.
This condition is equivalent to $f \in \gen[(C,\clProj{B})]{g}$.
\end{definition}

For a fixed clone $C$ on $A$, the $C$\hyp{}minor relation $\leq_C$ is a quasiorder (a reflexive and transitive relation) on $\mathcal{F}_{AB}$,
and it induces an equivalence relation $\equiv_C$ on $\mathcal{F}_{AB}$ ($f \equiv_C g$ if and only if $f \leq_C g$ and $g \leq_C f$)
and a partial order, also denoted by $\leq_C$, on $\mathcal{F}_{AB} / \mathord{\equiv_C}$ ($f / \mathord{\equiv_C} \leq_C g / \mathord{\equiv_C}$ if and only if $f \leq_C g$)
in the usual way.

\begin{lemma}[{\cite[Lemma~4.2]{Lehtonen-discmono}}]
\label{lem:Cminor}
Let $C$ be a clone on $A$, and let $F \subseteq \mathcal{F}_{AB}$.
The following statements are equivalent.
\begin{enumerate}[label={\textup{(\roman*)}}]
\item $F C = F$.
\item $F$ is a downset of the $C$\hyp{}minor quasiorder $(\mathcal{F}_{AB}, \mathord{\leq_C})$.
\item $F$ is a $(C,\clProj{B})$\hyp{}clonoid.
\item $F = \bigcup D$ for some downset $D$ of the $C$\hyp{}minor partial order $(\mathcal{F}_{AB} / \mathord{\equiv_C}, \mathord{\leq_C})$.
\end{enumerate}
\end{lemma}

For $k \in \IN \cup \{\infty\}$, the \emph{disjointness hypergraph of rank $k$} of a Boolean function $f$ is the hypergraph $G(f,k)$ whose set of vertices is $f^{-1}(1)$ and a set $S \subseteq f^{-1}(1)$ is a hyperedge if and only if $2 \leq \card{S} \leq k$ and $\bigwedge S = \vect{0}$.
In particular, $G(f,2)$ and $G(f,\infty)$ are called the \emph{disjointness graph} and the \emph{disjointness hypergraph} of $f$, respectively.

\begin{proposition}[{\cite[Proposition~8]{LehNes-clique}}]
\label{prop:Uk-hom}
Let $f, g \in \clOX$ and let $k \in \{2, 3, \dots, \infty\}$.
Then
$f \minor_{\clUk{k}} g$, i.e.,
$f \in \gen[(\clUk{k},\clIc)]{g}$,
if and only if $G(f,k) \to G(g,k)$.
\end{proposition}

\begin{proposition}
\label{prop:Uk}
For any $K \subseteq \clOX$, $\gen[(\clUk{2},\clUk{\infty})]{K} = \gen[(\clUk{2},\clIc)]{K}$.
\end{proposition}

\begin{proof}
The inclusion $\gen[(\clUk{2},\clIc)]{K} \subseteq \gen[(\clUk{2},\clUk{\infty})]{K}$ is clear.

In order to prove the converse inclusion, let $f \in \gen[(\clUk{2},\clUk{\infty})]{K} = \clUk{\infty} ( K \clUk{2} )$.
Then $f = \eta ( \varphi_1, \dots, \varphi_\ell )$ for some $\eta \in \clUk{\infty}$, $\varphi_1, \dots, \varphi_\ell \in K \clUk{2}$.
Because $\varphi = (\varphi_1, \dots, \varphi_\ell)$ maps $f^{-1}(1)$ into $\eta^{-1}(1)$ and because $\eta \in \clUk{\infty}$, there exists a $j \in \nset{\ell}$ such that $\varphi_j(\vect{a}) = 1$ for all $\vect{a} \in f^{-1}(1)$.
In other words, $f \minorant \varphi_j$.

Now, $\varphi_j = g ( h_1, \dots, h_m )$ for some $g \in K$, $h_1, \dots, h_m \in \clUk{k}$, i.e.,
$\varphi_j \in \gen[(\clUk{k},\clIc)]{g}$.
By Proposition~\ref{prop:Uk-hom}, $G(\varphi_j,k) \to G(g,k)$.
Because $f$ is a minorant of $\varphi_j$, it follows that $G(f,k)$ is an induced subgraph of $G(\varphi_j,k)$, and therefore the inclusion map $G(f,k) \to G(\varphi_j,k)$ is a homomorphism.
Consequently, $G(f,k) \to G(g,k)$, which implies, by Proposition~\ref{prop:Uk-hom}, that
$f \in \gen[(\clUk{k},\clIc)]{g} \subseteq \gen[(\clUk{k},\clIc)]{K}$.
\end{proof}

\begin{theorem}
\label{thm:U2Uinf}
For all clones $C_1$ and $C_2$ such that $C_1 \subseteq K_1$ and $C_2 \subseteq K_2$ for some $K_1 \in \{\clUk{2}, \clWk{2}\}$ and $K_2 \in \{\clUk{\infty}, \clWk{\infty}\}$,
there are an uncountably infinitude of $(C_1,C_2)$\hyp{}clonoids.
\end{theorem}

\begin{proof}
It suffices to prove the statement for $C_1 = \clUk{2}$ and $C_2 = \clUk{\infty}$.
It holds for the remaining pairs $(C_1,C_2)$ of clones by Proposition \ref{prop:Knid} and Lemma~\ref{lem:clonmon}.

By the results of \cite{LehNes-clique}, the $\clUk{k}$\hyp{}minor partial order $(\mathcal{F}_{AB} / \mathord{\equiv_{\clUk{k}}}, \mathord{\leq_{\clUk{k}}})$ is universal in the class of countable posets, even when restricted to $\clOX / \mathord{\equiv_{\clUk{k}}}$ (or even to $\clMc / \mathord{\equiv_{\clUk{k}}}$).
Thus, there exists a countably infinite antichain $Q$ in $\mathbf{P} = (\clOX / \mathord{\equiv_{\clUk{k}}}, \mathord{\minor_{\clUk{k}}})$.
Distinct subsets of $Q$ generate distinct downsets of $\mathbf{P}$.
Because the power set of $Q$ is uncountable, it follows that there are an uncountable infinitude of downsets of $\mathbf{P}$.

Because $\clUk{k} \subseteq \clOX$, it follows from
Lemmata~\ref{lem:a-pres} and \ref{lem:Cminor} that
the downsets of $\mathbf{P}$ are also downsets of $(\mathcal{F}_{AB} / \mathord{\equiv_{\clUk{k}}}, \mathord{\leq_{\clUk{k}}})$.
Moreover, again by Lemma~\ref{lem:Cminor}, the downsets of the $\clUk{k}$\hyp{}minor poset $(\mathcal{F}_{AB} / \mathord{\equiv_{\clUk{k}}}, \mathord{\leq_{\clUk{k}}})$ are essentially the same thing as $(\clUk{k},\clIc)$\hyp{}clonoids.
Consequently, there are an uncountably infinitude of $(\clUk{k},\clIc)$\hyp{}clonoids that are subsets of $\clOX$.
By Proposition~\ref{prop:Uk}, $\gen[(\clUk{k},\clIc)]{K} = \gen[(\clUk{k},\clUk{\infty})]{K}$ for all $K \subseteq \clOX$, and it follows that there are an uncountable infinitude of $(\clUk{k},\clUk{\infty})$\hyp{}clonoids as well.
This holds in particular when $k = 2$.
\end{proof}


\section{Finite clonoid lattices}

In this section, we show that for certain pairs $(C_1,C_2)$ of clones, there are only a finite number of $(C_1,C_2)$\hyp{}clonoids, and we explicitly describe such $(C_1,C_2)$\hyp{}clonoids.

\begin{theorem}
\label{thm:finite}
For all clones $C_1$ and $C_2$ such that $C_1 \supseteq K_1$ and $C_2 \supseteq K_2$ for some
\[
K_1 \in \{\clI, \clVo, \clLambdai\},
\quad
K_2 \in \{\clMcUk{\infty}, \clMcWk{\infty} \},
\]
there are only finitely many $(C_1,C_2)$\hyp{}clonoids.
\end{theorem}

\begin{proof}
It suffices to prove the statement for $(C_1,C_2) = (K_1,K_2)$, where
$(K_1,K_2)$ is one of the pairs
$(\clI,\clMcUk{\infty})$
and
$(\clVo,\clMcUk{\infty})$.
This is the content of Propositions \ref{prop:IMcUinf} and \ref{prop:VoMcUinf} below.
By Lemma~\ref{lem:clonmon}, it holds that for all clones $C_1$ and $C_2$ such that $C_1 \supseteq K_1$ and $C_2 \supseteq K_2$, there are only a finite number of $(C_1,C_2)$\hyp{}clonoids.
For the remaining pairs $(K_1,K_2)$ specified in the statement of the theorem, the result follows from Proposition~\ref{prop:Knid}.
\end{proof}

The remainder of this section is devoted to the proof of the finiteness of the lattices of $(\clI,\clMcUk{\infty})$- and $(\clVo,\clMcUk{\infty})$\hyp{}clonoids  (Propositions \ref{prop:IMcUinf} and \ref{prop:VoMcUinf}).
In fact, we will explicitly describe these clonoids.

\begin{proposition}
\label{prop:IMcUinf}
There are precisely 7 $(\clI,\clMcUk{\infty})$\hyp{}clonoids, and they are
$\clAll$, $\clM$, $\clMneg$, $\clVak$, $\clVako$, $\clVaki$, and $\clEmpty$.
These are also the $(\clI,\clMcWk{\infty})$\hyp{}clonoids.
\end{proposition}

\begin{proof}
\renewcommand{\gendefault}{(\clI,\clMcUk{\infty})}
We prove the claim about $(\clI,\clMcUk{\infty})$\hyp{}clonoids.
It turns out that dualization is a bijection on the set of all $(\clI,\clMcUk{\infty})$\hyp{}clonoids.
Consequently, by Proposition~\ref{prop:Knid}, these classes are also precisely the $(\clI,\clMcWk{\infty})$\hyp{}clonoids.

It is straightforward to verify that each one of the classes $\clAll$, $\clM$, $\clMneg$, $\clVak$, $\clVako$, $\clVaki$, and $\clEmpty$ is stable under right composition with $\clI$ and stable under left composition with $\clMcUk{\infty}$.
(For the left stability under $\clMcUk{\infty}$, it might be helpful to note that $\clMcUk{\infty}$ is a subclone of $\clM$, and hence $\clMcUk{\infty} \, \clM \subseteq \clM$ and $\clMcUk{\infty} \, \clMneg \subseteq \clMneg$.)

It remains to show that there are no further $(\clI,\clMcUk{\infty})$\hyp{}clonoids.
This is achieved by showing that every subset of $\clAll$ generates one of $\clAll$, $\clM$, $\clMneg$, $\clVak$, $\clVako$, $\clVaki$, and $\clEmpty$.
More precisely, for each one of these classes, say $K$, we show that for every subset $S$ of $K$ that is not included in any of the proper subclasses of $K$, as listed above, $\gen{S} = K$.
This will be done in the claims that follow.

\begin{claim}
\label{clm:IMcUinf:All}
If $f \in \clAll \setminus \clM$ and $g \in \clAll \setminus \clMneg$, then $\gen{f,g} = \clAll$.
\end{claim}

\begin{pfclaim}[Proof of Claim~\ref{clm:IMcUinf:All}]
We show first that $\clOmegaOne \subseteq \{f, g\} \clI$.
Because $f \in \clAll \setminus \clM$, there exist $\vect{a}$ and $\vect{b}$ such that $\vect{a} < \vect{b}$ and $f(\vect{a}) = 1 > 0 = f(\vect{b})$.
We see that the constant functions and the negated projections are in $\{ f \} \clI$ because
$\vak{1} = f(\vak{a_1}, \dots, \vak{a_n})$,
$\vak{0} = f(\vak{b_1}, \dots, \vak{b_n})$,
and
$\overline{\pr^{(m)}_i} = f(\alpha_1, \dots, \alpha_n)$, where
\[
\alpha_i =
\begin{cases}
\vak{0} & \text{if $a_i = b_i = 0$,} \\
\vak{1} & \text{if $a_i = b_i = 1$,} \\
\pr^{(m)}_i & \text{if $a_i = 1$ and $b_i = 0$,}
\end{cases}
\]
and $\vak{0}, \vak{1}, \pr^{(m)}_i \in \clI$.
In a similar way we get $\pr^{(m)}_j = g(\beta_1, \dots, \beta_n)$ for suitable $\beta_1, \dots, \beta_n \in \clI$.
Therefore $\clOmegaOne \subseteq \{f, g\} \clI$.

Now, let $\varphi \in \clAll$, say, of arity $n$, and define $\varphi' \colon \{0,1\}^{2n+2} \to \{0,1\}$,
\[
\varphi'(\vect{a}, \vect{b}, c, d) =
\begin{cases}
1 & \text{if $c = 1$ and ($d = 1$ or $w(\vect{a} \vect{b}) > n$),} \\
\varphi(\vect{a}) & \text{if $c = 1$, $d = 0$, $\vect{a} = \overline{\vect{b}}$,} \\
0 & \text{if $c = 0$ or ($c = 1$, $d = 0$, $w(\vect{a} \vect{b}) \leq n$, and $\vect{a} \neq \overline{\vect{b}}$).}
\end{cases}
\]
Clearly, $\varphi' \in \clMcUk{\infty}$ and $\varphi = \varphi'(\pr_1, \dots, \pr_n, \overline{\pr_1}, \dots, \overline{\pr_n}, \vak{1}, \vak{0}) \in \clMcUk{\infty} \, \clOmegaOne$.
Therefore, $\clAll \subseteq \clMcUk{\infty} \, \clOmegaOne \subseteq \clMcUk{\infty} (\{f, g\} \, \clI) = \gen{f,g} \subseteq \clAll$.
\end{pfclaim}

\begin{claim}
\label{clm:IMcUinf:M}
If $f \in \clM \setminus \clVak$, then $\gen{f} = \clM$.
\end{claim}

\begin{pfclaim}[Proof of Claim~\ref{clm:IMcUinf:M}]
Because $f \in \clM \setminus \clVak$, there exist $\vect{a}$ and $\vect{b}$ such that $\vect{a} < \vect{b}$ and $f(\vect{a}) = 0 < 1 = f(\vect{b})$.
In the same way as above, we get that $\clI \subseteq \{f\} \, \clI$.

Now, let $\varphi \in \clM$, say, of arity $n$, and define $\varphi' \colon \{0,1\}^{n+2} \to \{0,1\}$,
\[
\varphi'(\vect{a}, c, d) =
\begin{cases}
1 & \text{if $c = d = 1$,} \\
\varphi(\vect{a}) & \text{if $c = 1$, $d = 0$,} \\
0 & \text{if $c = 0$.}
\end{cases}
\]
Clearly, $\varphi' \in \clMcUk{\infty}$ and $\varphi = \varphi'(\pr_1, \dots, \pr_n, \vak{1}, \vak{0}) \in \clMcUk{\infty} \, \clI$.
Therefore, $\clM \subseteq \clMcUk{\infty} \, \clI \subseteq \clMcUk{\infty} (\{f\} \, \clI) = \gen{f} \subseteq \clM$.
\end{pfclaim}

\begin{claim}
\label{clm:IMcUinf:Mneg}
If $f \in \clMneg \setminus \clVak$, then $\gen{f} = \clMneg$.
\end{claim}

\begin{pfclaim}[Proof of Claim~\ref{clm:IMcUinf:Mneg}]
Because $f \in \clMneg \setminus \clVak$, there exist $\vect{a}$ and $\vect{b}$ such that $\vect{a} < \vect{b}$ and $f(\vect{a}) = 1 > 0 = f(\vect{b})$.
In the same way as above, we get that $\overline{\clI} \subseteq \{f\} \, \clI$.

Now, let $\varphi \in \clMneg$, say, of arity $n$, and define $\varphi' \colon \{0,1\}^{n+2} \to \{0,1\}$,
\[
\varphi'(\vect{a}, c, d) =
\begin{cases}
1 & \text{if $c = d = 1$,} \\
\varphi(\overline{\vect{a}}) & \text{if $c = 1$, $d = 0$,} \\
0 & \text{if $c = 0$.}
\end{cases}
\]
Clearly, $\varphi' \in \clMcUk{\infty}$ and $\varphi = \varphi'(\overline{\pr_1}, \dots, \overline{\pr_n}, \vak{1}, \vak{0}) \in \clMcUk{\infty} \, \overline{\clI}$.
Therefore, $\clMneg \subseteq \clMcUk{\infty} \, \overline{\clI} \subseteq \clMcUk{\infty} (\{f\} \, \clI) = \gen{f} \subseteq \clMneg$.
\end{pfclaim}

\begin{claim}
\label{clm:IMcUinf:const}
$\gen{\vak{0}, \vak{1}} = \clVak$, $\gen{\vak{0}} = \clVako$, $\gen{\vak{1}} = \clVaki$, $\gen{\clEmpty} = \clEmpty$.
\end{claim}

\begin{pfclaim}[Proof of Claim~\ref{clm:IMcUinf:Mneg}]
Clear.
\end{pfclaim}

All classes are exhausted, and the proof is complete.
\end{proof}

\begin{proposition}
\label{prop:VoMcUinf}
There are 13 $(\clVo,\clMcUk{\infty})$\hyp{}clonoids, and they are
$\clEmpty$, $\clVako$, $\clVaki$, $\clVak$,
$\clMo$, $\clMoneg$, $\clM$, $\clMneg$,
$\clOX$, $\clIX$, $\clOXCI$, $\clIXCO$, $\clAll$.
\end{proposition}

\begin{proof}
Firstly, we need to verify that the given classes are indeed $(\clVo,\clMcUk{\infty})$\hyp{}clonoids.
It suffices to do this only for the meet\hyp{}irreducible classes $\clAll$, $\clOXCI$, $\clIXCO$, $\clOX$, $\clIX$, $\clM$, and $\clMneg$, because the intersection of $(\clVo,\clMcUk{\infty})$\hyp{}clonoids is again a $(\clVo,\clMcUk{\infty})$\hyp{}clonoid.
This is mainly straightforward verification; note that $\clVo$ and $\clMcUk{\infty}$ are subclones of $\clAll$, $\clOX$, $\clM$, so the stability under right and left composition is clear.

We give here the proof of the fact that $\clMcUk{\infty} (\clOXCI) \subseteq \clOXCI$; the remaining cases are left as an exercise for the reader.
Let $\varphi \in \clMcUk{\infty} (\clOXCI)$.
Then $\varphi = f(g_1, \dots, g_n)$ for some $f \in \clMcUk{\infty}$ and $g_1, \dots, g_n \in \clOXCI$.
If $\varphi(\vect{0}) = 0$, then $\varphi \in \clOX \subseteq \clOXCI$ and we are done.
Assume that $\varphi(\vect{0}) = 1$.
Then $\vect{d} = (g_1(\vect{0}), \dots, g_n(\vect{0})) \in f^{-1}(1)$.
Because $g_i \in \clOXCI$ for all $i \in \nset{n}$, we have $d_i = 1$ if and only if $g_i \in \clVaki$, and it follows that for all $\vect{a} \in \{0,1\}^m$, $(g_1(\vect{a}), \dots, g_n(\vect{a})) \geq \vect{d}$.
Because $f$ is monotone, this implies that $(g_1(\vect{a}), \dots, g_n(\vect{a})) \in f^{-1}(1)$, and we conclude that $\varphi = \vak{1} \in \clVaki \subseteq \clOXCI$.

Secondly, we need to show that there are no further $(\clVo,\clMcUk{\infty})$\hyp{}clonoids.
We achieve this by showing that the $(\clVo,\clMcUk{\infty})$\hyp{}clonoid generated by any subset of $\clAll$ is one of the classes given in the statement of this proposition.
More precisely, we show that, for each class $K$ listed above, every subset of $K$ that is not included in any proper subclass of $K$ generates $K$.
This is done in the following claim.

\begin{claim}
\label{clm:VoMcUinf}
\leavevmode
\begin{enumerate}[label=\upshape{(\alph*)}]
\item\label{clm:VoMcUinf:Empty}
$\gen[(\clVc,\clMcUk{\infty})]{\clEmpty} = \clEmpty$.

\item\label{clm:VoMcUinf:C0}
For $f \in \clVako$, we have $\gen[(\clVc,\clMcUk{\infty})]{f} = \clVako$.

\item\label{clm:VoMcUinf:C1}
For $f \in \clVaki$, we have $\gen[(\clVc,\clMcUk{\infty})]{f} = \clVaki$.

\item\label{clm:VoMcUinf:C}
For $f \in \clVak \setminus \clVako$ and $g \in \clVak \setminus \clVaki$, we have $\gen[(\clVc,\clMcUk{\infty})]{f,g} = \clVak$.

\item\label{clm:VoMcUinf:M0}
For $f \in \clMo \setminus \clVako$, we have $\gen[(\clVc,\clMcUk{\infty})]{f} = \clMo$.

\item\label{clm:VoMcUinf:M0neg}
For $f \in \clMoneg \setminus \clVaki$, we have $\gen[(\clVc,\clMcUk{\infty})]{f} = \clMoneg$.

\item\label{clm:VoMcUinf:M}
For $f \in \clM \setminus \clMo$ and $g \in \clM \setminus \clVak$, we have $\gen[(\clVc,\clMcUk{\infty})]{f,g} = \clM$.

\item\label{clm:VoMcUinf:Mneg}
For $f \in \clMneg \setminus \clMoneg$ and $g \in \clM \setminus \clVak$, we have $\gen[(\clVc,\clMcUk{\infty})]{f,g} = \clMneg$.

\item\label{clm:VoMcUinf:OX}
For $f \in \clOX \setminus \clMo$, we have $\gen[(\clVc,\clMcUk{\infty})]{f} = \clOX$.

\item\label{clm:VoMcUinf:IX}
For $f \in \clIX \setminus \clMoneg$, we have $\gen[(\clVc,\clMcUk{\infty})]{f} = \clIX$.

\item\label{clm:VoMcUinf:OXCI}
For $f \in (\clOXCI) \setminus \clOX$ and $g \in (\clOXCI) \setminus \clM$, we have $\gen[(\clVc,\clMcUk{\infty})]{f,g} = \clOXCI$.

\item\label{clm:VoMcUinf:IXCO}
For $f \in (\clIXCO) \setminus \clIX$ and $g \in (\clIXCO) \setminus \clMneg$, we have $\gen[(\clVc,\clMcUk{\infty})]{f,g} = \clIXCO$.

\item\label{clm:VoMcUinf:All}
For $f \in \clAll \setminus (\clOXCI)$ and $g \in \clAll \setminus (\clIXCO)$, we have $\gen[(\clVc,\clMcUk{\infty})]{f,g} = \clAll$.
\end{enumerate}
\end{claim}

\begin{pfclaim}[Proof of Claim~\ref{clm:VoMcUinf}]
\ref{clm:VoMcUinf:Empty}
Clear.

\ref{clm:VoMcUinf:C0}
We have $f = \vak{0}$ and $\{\vak{0}\} \clVc = \clVako$, so $\gen[(\clVc,\clMcUk{\infty})]{f} = \clMcUk{\infty} ( \{f\} \clVc ) = \clMcUk{\infty} \clVako = \clVako$.

\ref{clm:VoMcUinf:C1}
We have $f = \vak{1}$ and $\{\vak{1}\} \clVc = \clVaki$, so $\gen[(\clVc,\clMcUk{\infty})]{f} = \clMcUk{\infty} ( \{f\} \clVc ) = \clMcUk{\infty} \clVaki = \clVaki$.

\ref{clm:VoMcUinf:C}
We have $f \in \clVaki$ and $g \in \clVako$.
Therefore, by \ref{clm:VoMcUinf:C0} and \ref{clm:VoMcUinf:C1}, we have
$\clVak = \clVako \cup \clVaki = \gen[(\clVc,\clMcUk{\infty})]{g} \cup \gen[(\clVc,\clMcUk{\infty})]{f} \subseteq \gen[(\clVc,\clMcUk{\infty})]{f,g} \subseteq \clVak$.

\ref{clm:VoMcUinf:M0}
Because $f$ is a non\hyp{}constant monotone function, we have $f(\vect{0}) = 0$ and $f(\vect{1}) = 1$.
Then $f(\id, \dots, \id) = \id$, and, in fact, for every $\gamma \in \clVo$, $f(\gamma, \dots, \gamma) = \gamma$, so $\clVo \subseteq \{f\} \clVo$.

We need to show that $\clMo \subseteq \gen[(\clVo,\clMcUk{\infty})]{f}$.
Let $\varphi \in \clMo \setminus \clVako$ with $\arity{\varphi} =: m$.
Define $\varphi' \colon \{0,1\}^{m+1} \to \{0,1\}$,
\begin{equation}
\varphi'(a_1, \dots, a_{m+1}) =
\begin{cases}
\varphi(a_1, \dots, a_m) & \text{if $a_m = 1$,} \\
0 & \text{otherwise.}
\end{cases}
\label{eq:phiM}
\end{equation}
By construction, $\varphi' \in \clMcUk{\infty}$, and we have
$\varphi = \varphi'(\pr^{(m)}_1, \dots, \pr^{(m)}_m, \mathord{\vee_m})
\in \clMcUk{\infty} \, \clVo \subseteq \clMcUk{\infty} ( \{f\} \clVo )
= \gen[(\clVo,\clMcUk{\infty})]{f}$.
Moreover,
$\vak{0} = \pr^{(1)}_1(\vak{0}) \in \clMcUk{\infty} \, \clVo \subseteq \clMcUk{\infty} ( \{f\} \clVo ) = \gen[(\clVo,\clMcUk{\infty})]{f}$.

\ref{clm:VoMcUinf:M0neg}
Because $f$ is a non\hyp{}constant antitone function, we have $f(\vect{0}) = 1$ and $f(\vect{1}) = 0$.
Then $f(\id, \dots, \id) = \neg$, and, in fact, for every $\gamma \in \clVo$, $f(\gamma, \dots, \gamma) = \overline{\gamma}$, so $\overline{\clVo} \subseteq \{f\} \clVo$.

We need to show that $\clMoneg \subseteq \gen[(\clVo,\clMcUk{\infty})]{f}$.
Let $\varphi \in \clMoneg \setminus \clVaki$ with $\arity{\varphi} =: m$.
Define $\varphi'' \colon \{0,1\}^{m+1} \to \{0,1\}$,
\begin{equation}
\varphi''(a_1, \dots, a_{m+1}) =
\begin{cases}
\varphi(\overline{a_1}, \dots, \overline{a_m}) & \text{if $a_m = 1$,} \\
0 & \text{otherwise.}
\end{cases}
\label{eq:phiMneg}
\end{equation}
By construction, $\varphi'' \in \clMcUk{\infty}$, and we have
$\varphi = \varphi''(\neg^{(m)}_1, \dots, \neg^{(m)}_m, \vak{1})
\in \clMcUk{\infty} \, \overline{\clVo} \subseteq \clMcUk{\infty} ( \{f\} \clVo )
= \gen[(\clVo,\clMcUk{\infty})]{f}$.
Moreover,
$\vak{1} = \pr^{(1)}_1(\vak{1}) \in \clMcUk{\infty} \, \overline{\clVo} \subseteq \clMcUk{\infty} ( \{f\} \clVo ) = \gen[(\clVo,\clMcUk{\infty})]{f}$.

\ref{clm:VoMcUinf:M}
We have $f \in \clVaki$, so $f = \vak{1}$, and $g$ is a nonconstant monotone function, so $g(\vect{0}) = 0$ and $g(\vect{1}) = 1$ and hence $\id = g(\id, \dots, \id)$; in fact, for every $\gamma \in \clVo$, $g(\gamma, \dots, \gamma) = \gamma$, so $\clVo \subseteq \{g\} \clVo$.
Consequently, $\clV \subseteq \{f, g\} \clVo$.

We need to show that $\clM \subseteq \gen[(\clVc,\clMcUk{\infty})]{f,g}$.
Let $\varphi \in \clM \setminus \clVak = \clMc$, and let $\varphi'$ be as defined in \eqref{eq:phiM}.
Then $\varphi' \in \clMcUk{\infty}$, and $\varphi = \varphi'(\pr^{(m)}_1, \dots, \pr^{(m)}_m, \vak{1}) \in \clMcUk{\infty} \, \clV \subseteq \clMcUk{\infty} ( \{f, g\} \clVo )$.
Moreover, for each $\vak{x} \in \clVak$, we have $\vak{x} = \id(\vak{x}) \in \clMcUk{\infty} \, \clV \subseteq \clMcUk{\infty} ( \{f, g\} \clVo )$.

\ref{clm:VoMcUinf:Mneg}
We have $f \in \clVako$, so $f = \vak{0}$, and $g$ is a nonconstant antitone function, so $g(\vect{0}) = 1$ and $g(\vect{1}) = 0$ and hence $\neg = g(\id, \dots, \id)$; in fact, for every $\gamma \in \clVo$, $g(\gamma, \dots, \gamma) = \overline{\gamma}$, so $\overline{\clVo} \subseteq \{g\} \clVo$.
Consequently, $\overline{\clV} \subseteq \{f, g\} \clVo$.

We need to show that $\clMneg \subseteq \gen[(\clVc,\clMcUk{\infty})]{f,g}$.
Let $\varphi \in \clMneg \setminus \clVak = \clMcneg$, and let $\varphi''$ be as defined in \eqref{eq:phiMneg}.
Then $\varphi'' \in \clMcUk{\infty}$, and $\varphi = \varphi''(\neg^{(m)}_1, \dots, \neg^{(m)}_m, \vak{1}) \in \clMcUk{\infty} \, \overline{\clV} \subseteq \clMcUk{\infty} ( \{f, g\} \clVo )$.
Moreover, for each $\vak{x} \in \clVak$, we have $\vak{x} = \id(\vak{x}) \in \clMcUk{\infty} \, \overline{\clV} \subseteq \clMcUk{\infty} ( \{f, g\} \clVo )$.

\ref{clm:VoMcUinf:OX}
We have $f(\vect{0}) = 0$, and there exist $\vect{u}$ and $\vect{v}$ such that $\vect{u} < \vect{v}$ and $f(\vect{u}) = 1$ and $f(\vect{v}) = 0$.
Without loss of generality,
\[
\vect{u} = (\underbrace{1, \dots, 1}_p, \underbrace{0, \dots, 0}_q, \underbrace{0, \dots, 0}_{n - p - q}),
\qquad
\vect{v} = (\underbrace{1, \dots, 1}_p, \underbrace{1, \dots, 1}_q, \underbrace{0, \dots, 0}_{n - p - q}).
\]
Let
\[
\alpha := f(\underbrace{\pr^{(2)}_1, \dots, \pr^{(2)}_1}_p, \underbrace{\pr^{(2)}_2, \dots, \pr^{(2)}_2}_q, \underbrace{\vak{0}, \dots, \vak{0}}_{n - p - q}).
\]
We have $\alpha \in \{f\} \clVo$ and $\alpha(0,0) = 0$, $\alpha(1,0) = 1$, $\alpha(1,1) = 0$.
We have $\id = \alpha(\pr^{(1)}_1, \vak{0}) \in \{\alpha\} \clVo \subseteq \{f\} \clVo$;
more generally, for any $\gamma \in \clVo$, $\gamma = \alpha(\gamma, \vak{0}) \in \{f\} \clVo$.
For $m \in \IN_{+}$ and $i \in \nset{m}$, let $\beta^{(m)}_i := \alpha(\mathord{\vee_m}, \pr^{(m)}_i)$.
We have $\beta^{(m)}_i \in \{\alpha\} \clVo \subseteq ( \{f\} \clVo ) \clVo = \{f\} \clVo$  and
\[
\beta^{(m)}_i(a_1, \dots, a_m) =
\begin{cases}
\overline{a_i} & \text{if $(a_1, \dots, a_m) \neq \vect{0}$,} \\
0 & \text{if $(a_1, \dots, a_m) = \vect{0}$.}
\end{cases}
\]

We need to show that $\clOX \subseteq \gen[(\clVc,\clMcUk{\infty})]{f}$.
Let $\varphi \in \clOX \setminus \clVako$ with $\arity{\varphi} =: m$, and define
$\widetilde{\varphi} \colon \{0,1\}^{2m+1} \to \{0,1\}$,
\begin{equation}
\begin{split}
\lhs
\widetilde{\varphi}(a_1, \dots, a_{2m+1}) =
\\ &
\begin{cases}
\varphi(a_1, \dots, a_m) & \text{if $a_{n+i} = \overline{a_i}$ for all $i \in \nset{m}$ and $a_{2m+1} = 1$,} \\
1 & \text{if $\card{\{ i \in \nset{2m} \mid a_i = 1 \}} > m$ and $a_{2m+1} = 1$,} \\
0 & \text{otherwise.}
\end{cases}
\end{split}
\label{eq:phiOX}
\end{equation}
By construction, $\widetilde{\varphi} \in \clMcUk{\infty}$, and we have
\[
\varphi = \widetilde{\varphi}(\pr^{(m)}_1, \dots, \pr^{(m)}_m, \beta^{(m)}_1, \dots, \beta^{(m)}_m, \mathord{\vee_m})
\in \clMcUk{\infty} (\{f\} \clVo).
\]
Moreover, $\vak{0} = \id(\vak{0}) \in \clMcUk{\infty} (\{f\} \clVo)$.

\ref{clm:VoMcUinf:IX}
We have $f(\vect{0}) = 1$, and there exist $\vect{u}$ and $\vect{v}$ such that $\vect{u} < \vect{v}$ and $f(\vect{u}) = 0$ and $f(\vect{v}) = 1$.
Without loss of generality,
\[
\vect{u} = (\underbrace{1, \dots, 1}_p, \underbrace{0, \dots, 0}_q, \underbrace{0, \dots, 0}_{n - p - q}),
\qquad
\vect{v} = (\underbrace{1, \dots, 1}_p, \underbrace{1, \dots, 1}_q, \underbrace{0, \dots, 0}_{n - p - q}).
\]
Let
\[
\alpha' := f(\underbrace{\pr^{(2)}_1, \dots, \pr^{(2)}_1}_p, \underbrace{\pr^{(2)}_2, \dots, \pr^{(2)}_2}_q, \underbrace{\vak{0}, \dots, \vak{0}}_{n - p - q}).
\]
We have $\alpha' \in \{f\} \clVo$ and $\alpha'(0,0) = 1$, $\alpha'(1,0) = 0$, $\alpha'(1,1) = 1$.
We have $\neg = \alpha'(\pr^{(1)}_1, \vak{0}) \in \{\alpha'\} \clVo \subseteq \{f\} \clVo$;
more generally, for any $\gamma \in \clVo$, $\overline{\gamma} = \alpha'(\gamma, \vak{0}) \in \{f\} \clVo$.
For $m \in \IN_{+}$ and $i \in \nset{m}$, let $\beta^{\prime (m)}_i := \alpha'(\mathord{\vee_m}, \pr^{(m)}_i)$.
We have $\beta^{\prime (m)}_i \in \{\alpha'\} \clVo \subseteq ( \{f\} \clVo ) \clVo = \{f\} \clVo$  and
\[
\beta^{\prime (m)}_i(a_1, \dots, a_m) =
\begin{cases}
a_i & \text{if $(a_1, \dots, a_m) \neq \vect{0}$,} \\
1 & \text{if $(a_1, \dots, a_m) = \vect{0}$.}
\end{cases}
\]

We need to show that $\clIX \subseteq \gen[(\clVc,\clMcUk{\infty})]{f}$.
Let $\varphi \in \clIX \setminus \clVaki$ with $\arity{\varphi} =: m$, and define
$\widetilde{\varphi}$ as in \eqref{eq:phiOX}.
By construction, $\widetilde{\varphi} \in \clMcUk{\infty}$, and we have
\[
\varphi = \widetilde{\varphi}(\beta^{\prime (m)}_1, \dots, \beta^{\prime (m)}_m, \neg^{(m)}_1, \dots, \neg^{(m)}_m, \vak{1})
\in \clMcUk{\infty} (\{f\} \clVo).
\]
Moreover, $\vak{1} = \id(\vak{1}) \in \clMcUk{\infty} (\{f\} \clVo)$.

\ref{clm:VoMcUinf:OXCI}
We have $f \in \clVaki$ and $g \in \clOX \setminus \clMo$.
We have shown above that $\gen[(\clVc,\clMcUk{\infty})]{f} = \clVaki$ and $\gen[(\clVc,\clMcUk{\infty})]{g} = \clOX$.
It follows that
$\clOXCI = \gen[(\clVc,\clMcUk{\infty})]{g} \cup \gen[(\clVc,\clMcUk{\infty})]{f} \subseteq \gen[(\clVc,\clMcUk{\infty})]{f,g} \subseteq \clOXCI$.

\ref{clm:VoMcUinf:IXCO}
We have $f \in \clVako$ and $g \in \clIX \setminus \clMoneg$.
We have shown above that $\gen[(\clVc,\clMcUk{\infty})]{f} = \clVako$ and $\gen[(\clVc,\clMcUk{\infty})]{g} = \clIX$.
It follows that
$\clIXCO = \gen[(\clVc,\clMcUk{\infty})]{g} \cup \gen[(\clVc,\clMcUk{\infty})]{f} \subseteq \gen[(\clVc,\clMcUk{\infty})]{f,g} \subseteq \clIXCO$.

\ref{clm:VoMcUinf:All}
We have $f \in \clIX \setminus \clVaki$ and $g \in \clOX \setminus \clVako$.
Thus, $f(\vect{0}) = 1$ and there is a $\vect{u}$ such that $f(\vect{u}) = 1$; without loss of generality, $\vect{u} = (1, \dots, 1, 0, \dots 0)$.
Then $\neg = f(\pr^{(1)}_1, \dots, \pr^{(1)}_1, \vak{0}, \dots, \vak{0}) \in \{f\} \clVo$.
In a similar way, we can show that $\id \in \{g\} \clVo$.
It follows that $\clVo \cup \overline{\clVo} \subseteq \{f, g\} \clVo$.

We need to show that $\clAll \subseteq \gen[(\clVc,\clMcUk{\infty})]{f,g}$.
Let $\varphi \in \clAll$, and define $\widetilde{\varphi}$ as in \eqref{eq:phiOX}.
By construction, $\widetilde{\varphi} \in \clMcUk{\infty}$ and
$\varphi = \widetilde{\varphi}(\pr^{(m)}_1, \dots, \pr^{(m)}_m, \neg^{(m)}_1, \dots, \neg^{(m)}_m, \vak{1}) \in \clMcUk{\infty} ( \{f, g\} \clVo )$.
\end{pfclaim}

This completes our proof.
\end{proof}

For the remaining pairs of clones $C_1$ and $C_2$ with $C_1 \in \{\clVo, \clLambdai\}$ and $C_2 \in \{ \clMcUk{\infty}, \clMcWk{\infty} \}$, the $(C_1,C_2)$\hyp{}clonoids can be obtained from Proposition~\ref{prop:VoMcUinf} by applying Proposition~\ref{prop:Knid}.

In order to determine the $(C_1,C_2)$\hyp{}clonoids for $C_1 \supseteq K_1$ and $C_2 \supseteq K_2$, where $K_1 \in \{\clI, \clVo, \clLambdai\}$ and $K_2 \in \{ \clMcUk{\infty}, \clMcWk{\infty} \}$, we simply observe that they are also $(K_1,K_2)$\hyp{}clonoids by Lemma~\ref{lem:clonmon}.
Therefore, it suffices to identify, for each $(K_1,K_2)$\hyp{}clonoid $K$ determined above in Propositions~\ref{prop:IMcUinf} and \ref{prop:VoMcUinf}, the clones with which $K$ is stable under left and right composition.
This work has actually been done already in an earlier paper \cite{Lehtonen-SM}, and we merely quote the relevant results here.

\begin{theorem}[{see \cite[Theorem~5.1]{Lehtonen-SM}}]
For each $(C_1,C_2)$-clonoid $K$, for $C_1 \in \{\clI, \clVo, \clLambdai\}$, $C_2 \in \{\clMcUk{\infty}, \clMcWk{\infty}\}$, as determined in Propositions \ref{prop:IMcUinf} and \ref{prop:VoMcUinf},
the clones $C_1^K$ and $C_2^K$ prescribed in Table~\ref{table:stable}
have the property that for every clone $C$, it holds that $KC \subseteq K$ if and only if $C \subseteq C_1^K$ and $CK \subseteq K$ if and only if $C \subseteq C_2^K$.
\end{theorem}

\begin{table}
\setlength{\tabcolsep}{3.9pt}
\begin{tabular}{ccc@{\quad\qquad}ccc}
\toprule
$K$ & $K C \subseteq K$ & $C K \subseteq K$ & $K$ & $K C \subseteq K$ & $C K \subseteq K$  \\
& iff $C \subseteq \ldots$ & iff $C \subseteq \ldots$ & & iff $C \subseteq \ldots$ & iff $C \subseteq \ldots$ \\
\midrule
$\clAll$ & $\clAll$ & $\clAll$ & & & \\
$\clOXCI$ & $\clOX$ & $\clM$ & $\clIXCO$ & $\clOX$ & $\clM$ \\
$\clXICO$ & $\clXI$ & $\clM$ & $\clXOCI$ & $\clXI$ & $\clM$ \\
$\clOX$ & $\clOX$ & $\clOX$ & $\clIX$ & $\clOX$ & $\clXI$ \\
$\clXI$ & $\clXI$ & $\clXI$ & $\clXO$ & $\clXI$ & $\clOX$ \\
$\clM$ & $\clM$ & $\clM$ & $\clMneg$ & $\clM$ & $\clM$ \\
$\clMo$ & $\clMo$ & $\clMo$ & $\clMoneg$ & $\clMo$ & $\clMi$ \\
$\clMi$ & $\clMi$ & $\clMi$ & $\clMineg$ & $\clMi$ & $\clMo$ \\
$\clVak$ & $\clAll$ & $\clAll$ & & & \\
$\clVako$ & $\clAll$ & $\clOX$ & $\clVaki$ & $\clAll$ & $\clXI$ \\
$\clEmpty$ & $\clAll$ & $\clAll$ & & & \\
\bottomrule
\end{tabular}
\bigskip
\caption{$(C_1,C_2)$-clonoids $K$, for $C_1 \in \{\protect\clI, \protect\clVo, \protect\clLambdai\}$, $C_2 \in \{\protect\clMcUk{\infty}, \protect\clMcWk{\infty}\}$, and their stability under right and left composition with clones of Boolean functions.}
\label{table:stable}
\end{table}


\section{Summary and final remarks}
\label{sec:summary}

\subsection{Known cardinalities of lattices of $(C_1,C_2)$\hyp{}clonoids of Boolean functions}
Table~\ref{table:card} summarizes what is currently known about the cardinalities of lattices of $(C_1,C_2)$\hyp{}clonoids of Boolean functions.
The rows correspond to source clones $C_1$, and the columns correcpond to target clones $C_2$.
The entry in row $C_1$ column $C_2$ indicates whether the lattice $\closys{(C_1,C_2)}$ is finite, countably infinite, or uncountable.
Some entries are known from earlier work, while some entries were determined in the current paper.
Moreover, some entries are implied by others by using Lemma~\ref{lem:clonmon}.
Namely, if $C_1 \subseteq C'_1$ and $C_2 \subseteq C'_2$, then $\closys{(C'_1,C'_2)}$ is a subposet of $\closys{(C_1,C_2)}$.
Consequently, if $\closys{(C_1,C_2)}$ is finite then so is $\closys{(C'_1,C'_2)}$; and if $\closys{(C'_1,C'_2)}$ is uncountable then so is $\closys{(C_1,C_2)}$.
In the cases when $\closys{(C_1,C_2)}$ was shown to be finite or countably infinite, the $(C_1,C_2)$\hyp{}clonoids have also been enumerated, although in some cases somewhat implicitly.
In the table, the footnote marks indicate references to the relevant results in the literature.

In order to keep the table relatively simple, we group together clones whose rows or columns would be identical according to our theoretical results.
Firstly, the intervals of target clones described in Corollary \ref{cor:Bfintervals} are grouped together.
Secondly, each (source or target) clone may be grouped together with its dual by Corollary~\ref{cor:duals}.
Moreover, Proposition~\ref{prop:card-Ic-Istar} indicates that the entries in columns $[\clIc,\clI]$ and $[\clIstar,\clOmegaOne]$ are equal for certain source clones.

The case when $C_1 = \clIc$
(the first row of Table~\ref{table:card}) is Sparks's Theorem~\ref{thm:Sparks}.
For each clone $C_2$ for which $\closys{(\clIc,C_2)}$ is finite, so is $\closys{(C_1,C_2)}$ for every clone $C_1$; this is the last column of Table~\ref{table:card} and these finite clonoid lattices were described in \cite[Theorems~4.1, 5.1]{Lehtonen-SM} and \cite[Theorem~6.1.1, Propositions~7.2.2, 7.3.1]{Lehtonen-nu}.

In the case when $C_2 \supseteq \clLc$ (the second last column of Table~\ref{table:card}), the clonoid lattices $\closys{(C_1,C_2)}$ were explicitly described and their cardinalities were determined in \cite[Theorems~6.1, 7.1]{CouLeh-Lcstability}.

The results for the case when $C_2 = \clIc$ follow from earlier work on so\hyp{}called $C$\hyp{}minors that was presented in a series of papers by the present author and his coauthors \cite{Lehtonen-ULM,LehNes-clique,LehSze-discriminator,LehSze-finite,LehSze-quasilinear}.
These results were translated into the language of $(C_1,C_2)$\hyp{}clonoids in \cite{Lehtonen-discmono}; moreover, the clonoid lattices $\closys{(C_1,C_2)}$ were explicitly described in the case when $C_1$ includes $\clMc$ or $\clSc$.

The main results of the current paper, Theorems~\ref{thm:uncountable}, \ref{thm:U2Uinf}, and \ref{thm:finite}, and also Proposition~\ref{prop:card-Ic-Istar} give several new entries to the table.


\newenvironment{outdent}
{\begin{list}{}{\leftmargin-3cm\rightmargin\leftmargin}\centering\item\relax}
{\end{list}\ignorespacesafterend}

\newcommand{\FCU}[2]{{#1}\makebox[0pt][l]{ #2}}
\newcommand{\FNfirst}[2][]{%
\newcounter{#2}%
\begingroup%
\ifthenelse{\equal{#1}{*}}{\renewcommand{\thefootnote}{(\arabic{footnote})}{}}{}%
\footnotemark\setcounter{#2}{\value{footnote}}%
\endgroup}

\newcommand{\FNother}[2][]{%
\begingroup%
\ifthenelse{\equal{#1}{*}}{\renewcommand{\thefootnote}{(\arabic{footnote})}{}}{}%
\footnotemark[\value{#2}]%
\endgroup}

\newcommand{\FNtxt}[2]{\footnotemark[\value{#1}] {#2}}
\newcommand{\FNseur}{\textsuperscript{\textsuperscript{$\ast$}}}
\newcommand{\FNcomma}{\textsuperscript{,}}

\begin{table}
\setcounter{footnote}{0}
\begin{outdent}\small
\begin{tabular}{r*{8}{l}}
     & & & & $\clMcUk{\infty}$ & & & & \\
     & & & & $\clMUk{\infty}$ & & & & $2 \leq \ell < \infty$ \\
     & & & $[\clVc,\clV]$ & $\clMcWk{\infty}$ & $\clTcUk{\infty}$ & $\clUk{\infty}$ & & $[\{\clSM, \clMcUk{\ell},$ \\
     & $[\clIc,\clI]$ & $[\clIstar,\clOmegaOne]$ & $[\clLambdac,\clLambda]$ & $\clMWk{\infty}$ & $\clTcWk{\infty}$ & $\clWk{\infty}$ & $[\clLc,\clL]$ & $\phantom{[\{}\clMcWk{\ell} \},\clAll]$ \\
\midrule
$\clIc$ & U\FNfirst{Sparks}\FNcomma\FNfirst{discmono-U}\FNcomma\FNfirst{X-U1}\FNcomma\FNfirst{X-U2}\FNcomma\FNfirst{IcIstar} & U\FNother{Sparks}\FNcomma\FNother{X-U1}\FNcomma\FNother{IcIstar} & U\FNother{Sparks}\FNcomma\FNother{X-U1}\FNcomma\FNother{X-U2} & U\FNother{Sparks}\FNcomma\FNother{X-U1}\FNcomma\FNother{X-U2} & U\FNother{Sparks}\FNcomma\FNother{X-U1}\FNcomma\FNother{X-U2} & U\FNother{Sparks}\FNcomma\FNother{X-U1}\FNcomma\FNother{X-U2} & C\FNother{Sparks}\FNcomma\FNfirst{LinSt} & F\FNother{Sparks}\FNcomma\FNfirst{SMnu} \\
$\clIo$, $\clIi$ & U\FNother{discmono-U}\FNcomma\FNother{X-U1}\FNcomma\FNother{X-U2}\FNcomma\FNother{IcIstar} & U\FNother{X-U1}\FNcomma\FNother{IcIstar} & U\FNother{X-U1}\FNcomma\FNother{X-U2} & U\FNother{X-U1}\FNcomma\FNother{X-U2} & U\FNother{X-U1}\FNcomma\FNother{X-U2} & U\FNother{X-U1}\FNcomma\FNother{X-U2} & C\FNother{LinSt} & F\FNother{SMnu} \\
$\clI$ & U\FNother{discmono-U}\FNcomma\FNother{X-U1} & U\FNother{X-U1} & U\FNother{X-U1} & F\FNfirst{X-F} & F\FNother{X-F} & F\FNother{X-F} & C\FNother{LinSt} & F\FNother{SMnu} \\
$\clIstar$ & U\FNother{discmono-U}\FNcomma\FNother{X-U1} & U\FNother{X-U1} & U\FNother{X-U1} & U\FNother{X-U1} & U\FNother{X-U1} & U\FNother{X-U1} & C\FNother{LinSt} & F\FNother{SMnu} \\
$\clOmegaOne$ & U\FNother{discmono-U}\FNcomma\FNother{X-U1} & U\FNother{X-U1} & U\FNother{X-U1} & F\FNother{X-F} & F\FNother{X-F} & F\FNother{X-F} & C\FNother{LinSt} & F\FNother{SMnu} \\
$\clVc$, $\clLambdac$ & U\FNother{discmono-U}\FNcomma\FNother{X-U1}\FNcomma\FNother{X-U2}\FNcomma\FNother{IcIstar} & U\FNother{X-U1}\FNcomma\FNother{IcIstar} & U\FNother{X-U1}\FNcomma\FNother{X-U2} & U\FNother{X-U2} & U\FNother{X-U2} & U\FNother{X-U2} & F\FNother{LinSt} & F\FNother{SMnu} \\
$\clVo$, $\clLambdai$ & U\FNother{discmono-U}\FNcomma\FNother{X-U1}\FNcomma\FNother{IcIstar} & U\FNother{X-U1}\FNcomma\FNother{IcIstar} & U\FNother{X-U1} & F\FNother{X-F} & F\FNother{X-F} & F\FNother{X-F} & F\FNother{LinSt} & F\FNother{SMnu} \\
$\clVi$, $\clLambdao$ & U\FNother{discmono-U}\FNcomma\FNother{X-U1}\FNcomma\FNother{X-U2}\FNcomma\FNother{IcIstar} & U\FNother{X-U1}\FNcomma\FNother{IcIstar} & U\FNother{X-U1}\FNcomma\FNother{X-U2} & U\FNother{X-U2} & U\FNother{X-U2} & U\FNother{X-U2} & F\FNother{LinSt} & F\FNother{SMnu} \\
$\clV$, $\clLambda$ & U\FNother{discmono-U}\FNcomma\FNother{X-U1} & U\FNother{X-U1} & U\FNother{X-U1} & F\FNother{X-F} & F\FNother{X-F} & F\FNother{X-F} & F\FNother{LinSt} & F\FNother{SMnu} \\
$[\clMcUk{\infty}, \clUk{2}]$ & \multirow{2}{*}{U\FNother{discmono-U}\FNcomma\FNother{X-U2}\FNcomma\FNother{IcIstar}} & \multirow{2}{*}{U\FNother{IcIstar}} & \multirow{2}{*}{U\FNother{X-U2}} & \multirow{2}{*}{U\FNother{X-U2}} & \multirow{2}{*}{U\FNother{X-U2}} & \multirow{2}{*}{U\FNother{X-U2}} & \multirow{2}{*}{F\FNother{LinSt}} & \multirow{2}{*}{F\FNother{SMnu}} \\
$[\clMcWk{\infty}, \clWk{2}]$ & & & & & & & & \\
$\clLc$ & U\FNother{discmono-U}\FNcomma\FNother{X-U1}\FNcomma\FNother{IcIstar} & U\FNother{X-U1}\FNcomma\FNother{IcIstar} & ? & ? & ? & ? & C\FNother{LinSt} & F\FNother{SMnu} \\
$\clLo$, $\clLi$ & U\FNother{discmono-U}\FNcomma\FNother{X-U1}\FNcomma\FNother{IcIstar} & U\FNother{X-U1}\FNcomma\FNother{IcIstar} & ? & ? & ? & ? & C\FNother{LinSt} & F\FNother{SMnu} \\
$\clLS$ & U\FNother{discmono-U}\FNcomma\FNother{X-U1} & U\FNother{X-U1} & ? & ? & ? & ? & C\FNother{LinSt} & F\FNother{SMnu} \\
$\clL$ & U\FNother{discmono-U}\FNcomma\FNother{X-U1} & U\FNother{X-U1} & ? & F\FNother{X-F} & F\FNother{X-F} & F\FNother{X-F} & C\FNother{LinSt} & F\FNother{SMnu} \\
$\clSM$ & U\FNother{discmono-U}\FNcomma\FNother{X-U2}\FNcomma\FNother{IcIstar} & U\FNother{IcIstar} & U\FNother{X-U2} & U\FNother{X-U2} & U\FNother{X-U2} & U\FNother{X-U2} & F\FNother{LinSt} & F\FNother{SMnu} \\
$[\clMc, \clM]$ & C\FNfirst{discmono-C} & C\FNother{discmono-C} & F\FNfirst{discmono-F-M} & F\FNother{discmono-F-M} & F\FNother{discmono-F-M} & F\FNother{discmono-F-M} & F\FNother{LinSt} & F\FNother{SMnu} \\
$[\clSc, \clAll]$ & F\FNfirst{discmono-F-S} & F\FNother{discmono-F-S} & F\FNother{discmono-F-S} & F\FNother{discmono-F-S} & F\FNother{discmono-F-S} & F\FNother{discmono-F-S} & F\FNother{LinSt} & F\FNother{SMnu} \\
\end{tabular}
\end{outdent}

\medskip
Glossary:
F -- finite;
C -- countably infinite;
U -- uncountable;
? -- unknown

\medskip
\begin{tablenotes}
\item \FNtxt{Sparks}{Sparks \cite[Theorem~1.3]{Sparks-2019}, Theorem~\ref{thm:Sparks}} \newline
\item \FNtxt{discmono-U}{\cite[Theorem~5.3]{Lehtonen-discmono}}  \newline
\item \FNtxt{X-U1}{Theorem \ref{thm:uncountable}} \newline
\item \FNtxt{X-U2}{Theorem \ref{thm:U2Uinf}} \newline
\item \FNtxt{IcIstar}{Proposition~\ref{prop:card-Ic-Istar}} \newline
\item \FNtxt{LinSt}{Couceiro, Lehtonen \cite[Theorems~6.1, 7.1]{CouLeh-Lcstability}} \newline
\item \FNtxt{SMnu}{\cite[Theorems~4.1, 5.1]{Lehtonen-SM}, \cite[Theorem~6.1.1, Propositions~7.2.2, 7.3.1]{Lehtonen-nu}} \newline
\item \FNtxt{X-F}{Theorem \ref{thm:finite}} \newline
\item \FNtxt{discmono-C}{\cite[Theorems~6.6, 6.7]{Lehtonen-discmono}} \newline
\item \FNtxt{discmono-F-M}{\cite[Propositions~6.12, 6.13, Theorem 6.14]{Lehtonen-discmono}} \newline
\item \FNtxt{discmono-F-S}{\cite[Theorems 7.2, 7.8, Propositions 7.4, 7.7]{Lehtonen-discmono}} \newline
\end{tablenotes}
\medskip
\caption{Cardinalities of $(C_1,C_2)$\hyp{}clonoid lattices.}
\label{table:card}
\end{table}

\subsection{Open problems and final remarks}
Table~\ref{table:card} contains a few question marks.
They designate the pairs $(C_1,C_2)$ of clones on $\{0,1\}$ for which the cardinality of the lattice $\closys{(C_1,C_2)}$ is not yet known.
This remains a topic for further research.

One might attempt to employ the proof technique of Section~\ref{sec:uncountable} to show that these clonoid lattices are uncountable (if this were indeed the case).
For this, it would be necessary to identify a different countably infinite family of functions than the functions $f_n$ and $q_n$ defined in Definitions~\ref{def:fn} and \ref{def:qn}.
The following lemma illustrates that the functions $f_n$ and $q_n$ fail to have the desired properties here.
In the proof, we use the notation $\mathcal{P}_k(S)$ for the set of all $k$\hyp{}element subsets of $S$.

\begin{lemma}
For any $n \geq 5$,
$f_{n+2} \in \gen[(\clLc,\clLambdac)]{f_n} = \clLambdac ( \{f_n\} \clLc )$
and
$q_{n+2} \in \gen[(\clLc,\clLambdac)]{q_n} = \clLambdac ( \{q_n\} \clLc )$.
\end{lemma}

\begin{proof}
Let $(\varphi_i)_{i \geq 5}$ be one of the families $(f_i)_{i \geq 5}$ or $(q_i)_{i \geq 5}$ of Boolean functions.

Let $n \in \IN_{+}$ with $n \geq 5$.
For each $S \in \mathcal{P}_3(\nset{n+2})$, let
$g^S \colon \{0,1\}^{n+2} \to \{0,1\}^n$ be the function
$g^S = (g^S_1, \dots, g^S_{n+2})$,
where
$g^S_1, \dots, g^S_{n-1}$ are the projections $\pr^{(n+2)}_i$ for $i \in \nset{n+2} \setminus S$ in some order
and
$g^S_n := \sum_{i \in S} x_i$.
Let $\varphi_S := \varphi_n \circ g^S$.

Now, define $\theta := \bigwedge_{S \in \mathcal{P}_3(\nset{n+2})} \varphi_S$.
Note that the $g^S_i$ are in $\clLc$, so $\varphi_S \in \{\varphi_n\} \clLc$, and therefore $\theta \in \clLambdac (\{\varphi_n\} \clLc) = \gen[(\clLc,\clLambdac)]{\varphi_n}$ by definition.
Our goal is to show that $\theta = \varphi_{n+2}$.

\begin{claim}
\label{clm:nonac-gS}
Let $\vect{a} \in \{0,1\}^{n+2}$.
\begin{enumerate}[label={\upshape(\roman*)}]
\item For all $S \in \mathcal{P}_3(\nset{n+2})$, $g^S(\vect{0}) = \vect{0}$.
\item For all $S \in \mathcal{P}_3(\nset{n+2})$, $g^S(\vect{1}) = \vect{1}$.
\item If $w(\vect{a}) = 1$, then for all $S \in \mathcal{P}_3(\nset{n+2})$, $w(g^S(\vect{a})) = 1$.
\item If $w(\vect{a}) = n+1$, then for all $S \in \mathcal{P}_3(\nset{n+2})$, $w(g^S(\vect{a})) = n-1$.
\item If $2 \leq w(\vect{a}) \leq n$, then for some $S \in \mathcal{P}_3(\nset{n+2})$, $w(g^S(\vect{a})) \notin \{1, n-1, n\}$.
\end{enumerate}
\end{claim}

\begin{pfclaim}[Proof of Claim~\ref{clm:nonac-gS}]
It is clear that $g^S(\vect{0}) = \vect{0}$ and $g^S(\vect{1}) = \vect{1}$ for all $S \in \mathcal{P}_3(\nset{n+2})$ because projections map $\vect{0}$ to $0$ and $\vect{1}$ to $1$, and $g^S_n(\vect{0}) = 0 + 0 + 0 = 0$ and $g^S_n(\vect{1}) = 1 + 1 + 1 = 1$.

Consider now the case when $w(\vect{a}) = 1$, i.e., $\vect{a} = \vect{e}_i$ for some $i \in \nset{n+2}$.
Let $S \in \mathcal{P}_3(\nset{n+2})$.
If $i \notin S$, then $g^S_n(\vect{e}_i) = 0$ and $g^S_j(\vect{e}_i) = 1$ for precisely one $j \in \nset{n-1}$.
If $i \in S$, then $g^S_n(\vect{e}_i) = 1$ and $g^S_j(\vect{e}_i) = 0$ for all $j \in \nset{n-1}$.
In either case, we have $w(g^S_n(\vect{e}_i)) = 1$.

The claim about the case when $w(\vect{a}) = n+1$ is proved in a similar way; now $\vect{a} = \overline{\vect{e}_i}$ for some $i \in \nset{n+2}$.

Assume now that $2 \leq w(\vect{a}) \leq n$.
We need to find an $S \in \mathcal{P}_3(\nset{n+2})$ such that $w(g^S(\vect{a})) \notin \{1, n-1, n\}$.
Let $A := \{ \, i \in \nset{n+2} \mid a_i = 1 \, \}$.

If $w(\vect{a}) = 2$, then choose $S \in \mathcal{P}_3(\nset{n+2})$ such that $A \subseteq S$.
Then $g^S(\vect{a}) = \vect{0}$, and hence $w(g^S(\vect{a})) = 0$, because the projections $g^S_1, \dots, g^S_{n-1}$ map $\vect{a}$ to $0$, and $g^S_n(\vect{a}) = 1 + 1 + 0 = 0$.

If $w(\vect{a}) = 3$, then choose $S \in \mathcal{P}_3(\nset{n+2})$ such that $S \subseteq \nset{n+2} \setminus A$.
Then $w(g^S(\vect{a})) = 3$ because exactly three of the projections $g^S_1, \dots, g^S_{n-1}$ map $\vect{a}$ to $1$, and $g^S_n(\vect{a}) = 0 + 0 + 0 = 0$.
Because $n \geq 5$, $3 \notin \{1, n-1, n\}$.

If $4 \leq w(\vect{a}) \leq n$, then choose $S \in \mathcal{P}_3(\nset{n+2})$ such that $\card{S \cap A} = 2$.
Then $w(g^S(\vect{a})) = w(\vect{a}) - 2 \notin \{1, n-1, n\}$ because exactly $w(\vect{a}) - 2$ of the projections $g^S_1, \dots, g^S_{n-1}$ map $\vect{a}$ to $1$, and $g^S_n(\vect{a}) = 1 + 1 + 0 = 0$.
\end{pfclaim}

Claim~\ref{clm:nonac-gS} shows that
for all $S \in \mathcal{P}_3(\nset{n+2})$,
$g^S$ maps the true points of $\varphi_{n+2}$ to true points of $\varphi_n$,
and, moreover,
for every false point $\vect{a}$ of $\varphi_{n+2}$, there is an $S \in \mathcal{P}_3(\nset{n+2})$ such that $\varphi_n(g^S(\vect{a})) = 0$.
Consequently,
\[
\theta(\vect{a}) = \bigwedge_{S \in \mathcal{P}_3(\nset{n+2})} \varphi_n(g^S(\vect{a})) = \varphi_{n+2}(\vect{a})
\]
for all $\vect{a} \in \{0,1\}^{n+2}$, which shows that $\theta = \varphi_{n+2}$, as claimed.
\end{proof}

The work reported here concerns $(C_1,C_2)$\hyp{}clonoids of Boolean functions.
A potentially fruitful direction for future research is to generalize these results to clonoids on arbitrary (finite) sets.


\section*{Acknowledgments}

The author would like to thank Sebastian Kreinecker for inspiring discussions.


\end{document}